\documentclass[11pt]{article}
\usepackage{graphicx,wrapfig}
\graphicspath{figures/}
\usepackage{flafter}
\usepackage{leftidx}
\usepackage{mathrsfs}
\usepackage{amssymb,amsmath}
\usepackage{bbding}
\usepackage{fancyhdr}
\usepackage{mathrsfs}
\usepackage{color}
\usepackage{stmaryrd}
\usepackage{amsfonts}
\usepackage{latexsym}
\usepackage{psfrag}
\usepackage{graphicx,subfigure}
\usepackage{fancyhdr,graphicx}
\usepackage{multicol}
\usepackage{dsfont}
\usepackage{bbm}
\usepackage{booktabs}
\usepackage[center]{caption2}
\usepackage{cite}
\usepackage{epstopdf}
\usepackage{diagbox}
\usepackage{booktabs}
\usepackage [utf8]{inputenc}
\usepackage{multirow}
\usepackage{enumerate}
\usepackage{enumitem}
\usepackage{algpseudocode,algorithm,algorithmicx}
\usepackage{appendix}
\usepackage{epstopdf}

\allowdisplaybreaks
\renewcommand{\leq}{\leqslant}

\date{}
\newtheorem{theorem}{Theorem}[section]
\newtheorem{lemma}{Lemma}[section]

\newtheorem{example}{Example}[section]

\numberwithin{equation}{section}

\newcommand{\zd}{\,\mathrm{d}}
\newcommand{\diff}{\triangledown_{\tau}}

\newcommand{\abs}[1]{\left|#1\right|}

\newcommand{\bra}[1]{\left(#1\right)}
\newcommand{\brab}[1]{\big(#1\big)}
\newcommand{\braB}[1]{\Big(#1\Big)}

\newcommand{\myinner}[1]{\left\langle#1\right\rangle}
\newcommand{\myinnerb}[1]{\big\langle#1\big\rangle}

\newcommand{\mynorm}[1]{\left\|#1\right\|}
\newcommand{\mynormb}[1]{\big\|#1\big\|}

\newcommand{\zhur}[1]{{\color{black}#1}}

\begin{document}
\title{Energy stable and $L^2$ norm convergent BDF3 scheme for the Swift–Hohenberg equation}
\author{Xuan Zhao\thanks{Corresponding author. School of
Mathematics, Southeast University, Nanjing 210096, P. R. China (xuanzhao11@seu.edu.cn).}
\quad Ran Yang\thanks{School of Mathematics, Southeast University, Nanjing 210096, P. R. China (220191527@seu.edu.cn).}
\quad Zhongqin Xue \thanks{School of
Mathematics, Southeast University, Nanjing 210096, P. R. China. (zqxue@seu.edu.cn).}
\quad Hong Sun\thanks{School of Mathematics, Southeast University, Nanjing 210096; Department of Mathematics and Physics, 
Nanjing Institute of Technology, Nanjing 211167, P. R. China (sunhongzhal@126.com).}}

\date{}
\maketitle
\normalsize
\begin{abstract}

A fully discrete implicit scheme is proposed for the Swift-Hohenberg model, combining the third-order backward differentiation formula (BDF3) for the time discretization and the second-order finite difference scheme for the space discretization. Applying the Brouwer fixed-point theorem and the positive definiteness of the convolution coefficients of BDF3, the presented numerical algorithm is proved to be uniquely solvable and unconditionally energy stable, further, the numerical solution is shown to be bounded in the maximum norm.  The proposed scheme is rigorously proved to be convergent  in $L^2$ norm  by the discrete orthogonal convolution (DOC) kernel, which transfer the four-level-solution form into the three-level-gradient form for the approximation of the temporal derivative.  Consequently, the error estimate for the numerical solution is established by utilization of  the discrete Gronwall inequality. Numerical examples in 2D and 3D cases are provided to support the theoretical results.

\end{abstract}

\section{Introduction}\setcounter{equation}{0}
In this work, a fully discrete implicit scheme is proposed for the following Swift-Hohenberg equation with the periodic boundary conditions 
\begin{align}\label{def:the SH equation}
	u_t=-(1+\Delta)^2u-f(u),
\end{align}
where $f(u)=u^3-\mathrm{g}u^2-\epsilon u,$ $\epsilon$ and $\mathrm{g}$ are positive contants with physical significance. A closed two-dimensional domain $\Omega \subset \mathbb{R}^2$ is considered for the theoretical part.

The Swift-Hohenberg equation is a nonlinear fourth-order partial differential equation, which describes the effects of thermal fluctuations on the Rayleigh-B$\acute{\rm e}$nard instability\cite{Swift_PRA1977}. As one of the paradigms of nonlinear dynamical system, the Swift-Hohenberg equation has been widely applied in physics\cite{Li_SA2020}, materials science\cite{Wise_SIAMJNA2009}, dynamics of ecological systems\cite{Tlidi_PD2020}, and many other nonlinear fields\cite{Stefanovic_PRE2009,Rosa_PA2000,Hutt_PD2008}.

The energy of the equation\eqref{def:the SH equation} is defined by the Lyapunov energy functional
\begin{align}\label{def:continuous energy}
	E(u):=\int_{\Omega}\braB{\frac{1}{2}u(1+\Delta)^2 u+\frac{1}{4}u^4-\frac{\mathrm{g}}{3}u^3-\frac{\epsilon}{2} u^2} \zd\mathbf{x},
\end{align}
and is nonincreasing in time
\begin{align}\label{def: continuous energy dissipation law}
	\frac{d}{dt}E(u)=-\int_{\Omega}\left( \frac{du}{dt}\right)^2\zd\mathbf{x} \le 0.
\end{align}

Bifurcation analysis\cite{Choi_DCDSSB2017} and pattern selection\cite{Peletier_PD2004} of the equilibrium solutions for the Swift-Hohenberg equation are theoretically analyzed. 
Shi and Han\cite{Shi_DCDSSS2020} obtained a reversible homoclinic solution approaching to a periodic solution of the Swift-Hohenberg equation, and the existence of nontrivial periodic solutions was proved by Lai and Zhang\cite{Lai_AML2020}. 
In \cite{Gao_CPAA2020}, the effects of rapid oscillations of the forcing term on the long-time behaviour of the solutions are studied. 
We also refer the reader to the recent article \cite{Kostin_MV2018} on nonstationary case and the references therein.

Recently, the efficient numerical methods for the Swift-Hohenberg equation are considered since the analytically solutions cannot be generally obtained. A fully discrete discontinuous Galerkin scheme is proposed in \cite{Liu_JSC2019}, where the time step as well as the mesh size are irrespective. The scheme is proved uniquely solvable and unconditionally energy stable. 
A second-order energy-stable time-integration method that suppresses numerical instabilities is presented by Sarmiento et al.\cite{Sarmiento_JCAM2018} and a detailed proof of the unconditional energy stability is provided.
Liu\cite{Liu_NA2020} considered two linear, second-order and unconditionally energy stable schemes by linear invariant energy quadratization and modified scalar auxiliary variable approaches. 
Dehghan et al.\cite{Dehghan_IJNMHFF2019} combined the proper orthogonal decomposition approach and the local discontinuous Galerkin technique, and discussed the energy stability. Sun et al.\cite{sunBDF2} proposed an adaptive BDF2 scheme for the Swift-Hohenberg equation, and proved the energy stability and the convergence. Besides, readers can refer to the references \cite{SuCAM2019,LEE20171885,DEHGHAN201749,GOMEZ20124930,Wang2011,Zhao2013,LEE2020112815,Wang2020,Zhou2019} for some related numerical methods. However, in most existing works, the convergence of the numerical algorithms for Swift-Hohenberg equation is rarely considered. 

As is known that the Swift-Hohenberg equation takes a long-time to reach the steady state, the key challenge in design the numerical methods is the improvement for the computational efficiency of the numerical schemes. High order time-stepping strategies and adaptive time-stepping algorithms are always suitable choice for developing the efficient numerical schemes. In the literature, we name the following methods in simplicity for the discretization of Swift-Hohenberg equation in temporal direction

\begin{itemize}
	\item[1)]  Crank-Nicolson type of scheme (second-order)\cite{Christov_MCM2002}
		\begin{align*}
			\dfrac{u^{k+1}-u^k}{\tau}&=D\left( -\frac{\partial^4}{\partial x^4}-\frac{\partial^4}{\partial y^4}-2k^2\frac{\partial^2}{\partial x^2}-2k^2\frac{\partial^2}{\partial y^2}-2\frac{\partial^4}{\partial x^2\partial y^2}-k^4\right)\dfrac{u^{k+1}+u^k}{2}\\
			&-\dfrac{U(u^{n+1})-U(u^{n})}{u^{n+1}-u^n}
		\end{align*}	
			where $U(u)$ stands for the potential of the nonlinear force upon the system. 
	\item[2)] Convex splitting method (first-/second-order) \cite{Lee_CMAME2019}
	\begin{itemize}
		\item[-]The first order convex splitting
		\begin{align*}
			\dfrac{u^{k+1}-u^k}{\tau}=-(u^{k+1})^3+\mathrm{g}(u^{k+1})^2-(A+(1+\Delta)^2)u^{k+1}+(A+\epsilon)u^k,
		\end{align*}
		\item[-]The second order convex splitting
		\begin{align*}
			\dfrac{u^{k+1}-u^k}{\tau}=-\left( \chi(u^{k+1},u^{k})+(A+(1+\Delta)^2)\left(\dfrac{u^{k+1}+u^k}{2}\right)-(A+\epsilon)\dfrac{3u^{k}-u^{n-1}}{2}\right),
		\end{align*}
		where $\chi(u,v)=\dfrac{u^2+v^2}{2}\dfrac{u+v}{2}-\mathrm{g}\dfrac{u^2+uv+v^2}{3}$.
	\end{itemize}
		\item[3)]Semi-implicit Euler scheme (first-order) \cite{Zhang_AAMM2018}
			\begin{align*}
			\dfrac{u^{k+1}-u^k}{\tau}+A(u^{k+1}-u^k)=2\Delta u^k-\Delta^2u^{k+1}-(1-\epsilon)u^{k+1}-(u^k)^3,
			\end{align*}where $A(u^{k+1}-u^k)$ is an extra artificial term to preserve the energy stability.		
			   \item[4)] Multiple SAV scheme based on the Crank-Nicolson method (second-order)\cite{Liu_NA2020}
		\begin{align*}
		\dfrac{u^{k+1}-u^k}{\tau}&=M\Delta \mu^{k+\frac{1}{2}},\\
		\mu^{k+\frac{1}{2}}&=(1+\Delta)^2\left(\dfrac{u^{k+1}+u^{k}}{2} \right)+\dfrac{r^{k+1}+r^k}{2\sqrt{\tilde{E}(\tilde{u}^{k+1/2})}}U(\tilde{u}^{k+1/2}) \\
		&-\dfrac{m^{k+1}+m^k}{2\sqrt{E_0(\tilde{u}^{k+1/2})}}V(\tilde{u}^{k+1/2}),\\
		\dfrac{r^{k+1}-r^k}{\tau}&=\dfrac{1}{2\sqrt{\tilde{E}(\tilde{u}^{k+1/2})}}\int_{\Omega}U(\tilde{u}^{k+1/2})\dfrac{u^{k+1}-u^k}{\tau}\zd\mathbf{x},\\
		\dfrac{m^{k+1}-m^k}{\tau}&=\dfrac{1}{2\sqrt{E_0(\tilde{u}^{k+1/2})}}\int_{\Omega}V(\tilde{u}^{k+1/2})\dfrac{u^{k+1}-u^k}{\tau}\zd\mathbf{x},\\				
		\end{align*}
	    where $E_0(u)=\int_{\Omega}\frac{\mathrm{g}}{6}u^4+\frac{\mathrm{g}}{6}u^2+\frac{\epsilon}{6}u^4>0$, $\tilde{E}(u)=\int_{\Omega}\frac{1}{4}u^4-\frac{\mathrm{g}}{3}u^3+\frac{\mathrm{g}}{6}u^4+\frac{\mathrm{g}}{6}u^2>0$, $U(u)=\frac{\delta\tilde{E}}{\delta u}$, $V(u)=\frac{\delta E_0}{\delta u}$, $r(t)=\sqrt{\tilde{E}(u)}$ and $m(t)=\sqrt{E_0(u)}$ are two scalar auxiliary variables, $\tilde{u}^{k+1/2}$ is any explicit $O(\tau^2)$ approximation for $u(t^{k+1/2})$. 	
		\item[5)] Invariant Energy Quadratization(IEQ) method (first-/second-order)\cite{Liu_JSC2019}
		\begin{align*}
			\left(\dfrac{u_h^{k+1}-u_h^k}{\tau},\phi \right)&=-A(\phi,q_h^{k+1})-(H(u_h^k)U^{k+1},\phi),\\
			 (q_h^k,\psi)&=A(u_h^k,\psi),~~~~~~~\text{for}~\forall \phi,\psi~~\text{in the space of piecewise polynomials},	
		\end{align*}
		where $H(u)=\dfrac{\Phi'(u)}{\sqrt{\Phi(u)+B}}$ for the constant $B>0,$ $\Phi=\frac{1}{4}u^4-\frac{\mathrm{g}}{3}u^3-\frac{\epsilon}{2}u^2,$ $U=\sqrt{\Phi(u)+B}$, besides, $A(\cdot,\cdot)$ is a biliner operator corresponding to the operator $-(\Delta+\frac{a}{2})$. Once the $q_h^{k+1}$, $U^{k+1}$ and $H(u_h^k)$ are replaced by $(q_h^{k+1}+q_h^{k})/2$, $(U^{k+1}+U_h^k)/2$ and $H(u_h^{k,*})$ respectively, with $u_h^{k,*}=\frac{3}{2}u_h^k-\frac{1}{2}u_h^{k-1}$, a second order discretization in time is obtained.
	
		\item[6)] Adaptive BDF2 scheme (second-order)\cite{sunBDF2}
			\begin{align*}
				\dfrac{1+2r_k}{\tau_k(1+r_k)}(u^{k}-u^{k-1})-\dfrac{r_k^2}{\tau_k(1+r_k)}(u^{k-1}-u^{k-2})=-(1+\Delta)^2u^k-f(u^k)	
			\end{align*}
		where $\tau_k=t_k-t_{k-1}$ is the time step and $r_k=\tau_k-\tau_{k-1}$. The adaptive time step allows rather larger time step than the uniform time discretizations.
\end{itemize}

As is seen that most of the existing approaches proposed in the literatures keeps first-/second-order accurate in time direction. Whereas, we consider a third order time-integration approach, namely the BDF3 formula\cite{Hairer_HSV1996}, to solve the Swift-Hohenberg equation for the time discretization with central difference approximation in space. This third-order accurate method has very satisfactory stability properties\cite{Kim_JCP2020}, and was found to be the most efficient in terms of the computational cost for a given accuracy level compared to the lower-order schemes in \cite{Zhang_AMM2013}.
However, it is known that the $L^2$-norm stability and convergence for the BDF3 formula are difficult to obtain as the scheme is not A-stable. Relying on the equivalence between A-stability and G-stability\cite{Dahlquist_BIT1978}, some useful tools for the numerical analysis of the BDF3 scheme were proposed\cite{Liu_SIAMJNA2013, Nevanlinna_NFAO1981}. Furthermore, Liao et al.\cite{Liao_2021} achieved a novel yet straightforward discrete energy analysis for the BDF3 schemes using the discrete orthogonal convolution (DOC) kernel technique, and the $L^2$-norm stability and convergence of the BDF3 scheme has already been rigorously proved for the linear reaction-diffusion equation.

We focus on developing a fully discrete implicit scheme for the Swift-Hohenberg equation. By virtue of the DOC kernel, the optimal $L^2$ error estimate is carried out for the proposed scheme, and the scheme is rigorously proved to be third order in time. To the best of our knowledge, this is the first time that the $L^2$ norm convergence of BDF3 method is obtained for the Swift-Hohenberg equation. 
Besides, by applying the Brouwer fixed-point theorem and other analytical methods, the unique solvability and the energy stability of the proposed scheme are also provided.

The remainder of the paper is organized as follows. In Section 2, with the introduction of some notations, the construction of the BDF3 difference scheme for the Swift-Hohenberg equation \eqref{def:the SH equation} is shown in detail. Section 3 presents the results that the proposed scheme is uniquely solvable and maintains the properties of global energy stability. In Section 4, using the DOC kernel proposed in \cite{Liao_MA2020}, the $L^2$-norm convergence of the BDF3 scheme is rigorously proved. Numerical tests of 2D and 3D are performed in Section 5 to verify the accuracy, the stability and the efficiency of the proposed scheme. Some concluding remarks are given in Section 6.

\section{Construction of the BDF3 difference scheme}
\setcounter{equation}{0}
In this section, we set up a fully discrete scheme for the Swift-Hohenberg equation described as follows
\begin{align}\label{def:the SH equation1}
	&u_t=-(1+\Delta)^2u-f(u),					&~~(x,y)\in \Omega,  ~0<t\le T,\\
	&u(0,y,t)=u(1,y,t),~u(x,0,t)=u(x,1,t), 	    &~~(x,y) \in \Omega, ~0<t\le T,
	\label{def:the SH equation2}\\
	&u(x,y,0)=\varphi_0(x,y), 					&~~(x,y) \in \overline{\Omega},
	\label{def:the SH equation3}
\end{align}
where $\Omega=(0,1)^2$, $\overline{\Omega}=[0,1]^2$, the nonlinear term $f(u)=u^3-\mathrm{g}u^2-\epsilon u$ with two positive physical parameters $\mathrm{g}$ and $\epsilon$, and $\varphi_0$ is a known smooth function.

Now we consider the difference schemes of \eqref{def:the SH equation1}-\eqref{def:the SH equation3}. For the discretization of time direction, the BDF3 formula is applied. Take a positive integer N, and denote $\tau=\frac{T}{N}$, $t_k=\tau k$ for $0\le k\le N$, then the well-known BDF3 formula can be described as \\
\begin{align}\label{def:BDF3}
D_3v^k:=\frac{11}{6\tau}\diff v^k - \frac{7}{6\tau}\diff v^{k-1} +\frac{1}{3\tau}\diff v^{k-2}, \qquad\text{ for $k \ge 3$},
\end{align}
where $\diff v^k:=v^k-v^{k-1}$ for any sequence $v^k=v(t_k)$.
Referring to \cite{Liao_MA2020}, the formula \eqref{def:BDF3} can be rewritten as a discrete convolution summation,
\begin{align}\label{CBDF3}
	D_3v^k:=\sum_{j=1}^kb_{k-j}\diff v^j,
	\quad \text{ $k\ge3$},
\end{align}
where the convolution kernels $b_{k-j}$ are defined by 
\begin{align}\label{b_{n-k}}
	b_{0}:=\frac{11}{6\tau},\quad
	b_{1}:=-\frac{7}{6\tau}\quad \text{and} \quad
	b_{2}:=\frac{1}{3\tau},\quad \text{togeter with} \quad
	b_{j}:=0,\quad \mathrm{for}\quad 3\le j\le k-1.
\end{align}

Since the BDF3 formula requires three previous levels of unknowns, extra schemes should be used to compute $u^1$ and $u^2$. Taylor expansion leads to 
$$\frac{1}{2}[u_t(t_1)+u_t(t_0)]=\frac{1}{\tau}[u(t_1)-u(t_0)]+O(\tau^2),$$
thus, a second-order scheme can be obtained
\begin{align}\label{def:u^1}
	D_3u^1=\frac{2}{\tau}\diff u^1,
\end{align}
where $u^0:=u(t_0)+\frac{\tau}{2}u_t(t_0)$. 
Then, the BDF2 formula is used to calculate $u^2$ as 
\begin{align}\label{def:u^2}
	D_3u^2=-\frac{1}{2\tau}\diff u^1+\frac{3}{2\tau}\diff u^2.
\end{align}

For the spatial direction discretization, let $M$ be a positive integer, the spatial lengths $h_x=h_y=h:=L/M$, the grid points $x_i=ih$, $y_j=jh$ and ${\mathrm{x}_h}=(x_i, y_j).$
The discrete spatial grid
$\Omega_h:=\big\{{\mathrm{x}_h}=(x_i, y_j)~|~1\le i, j\le M-1\big\}$
and $\bar{\Omega}_h:=\big\{{\mathrm{x}_h}=(x_i, y_j)~|~0\le i, j\le M\big\}.$
Consider the periodic function space
$$\mathcal{V}_h:=\big\{v_h=v({\mathrm{x}_h})~|~{\mathrm{x}_h}\in \bar{\Omega}_h
\;\text{and $v_h$ is periodic in each direction}\big\}.$$
Given a grid function $v\in \mathcal{V}_h,$ introduce the following notations
$\delta_xv_{i+\frac{1}{2},j}=(v_{i+1,j}-v_{ij})/h$, and $\delta^2_xv_{ij}=(\delta_xv_{i+\frac{1}{2},j}-\delta_xv_{i-\frac{1}{2},j})/h.$ The discrete notations $\delta_yv_{i,j+\frac{1}{2}}$ and $\delta^2_yv_{ij}$ can be defined similarly. Also, the discrete Laplacian operator $\Delta_hv_{ij}=\delta^2_xv_{ij}+\delta^2_yv_{ij}$ and the discrete gradient vector $\nabla_hv_{ij}=(\delta_xv_{ij},~\delta_yv_{ij})^T$.

Considering Eq. \eqref{def:the SH equation1} at the grid point ($\mathrm{x}_h, t_n$), we have
\begin{align*}
u_t(\mathrm{x}_h, t_n)+(1+\Delta)^2u(\mathrm{x}_h, t_n)+f(u(\mathrm{x}_h, t_n))=0.
\end{align*}
It follows from \eqref{CBDF3}-\eqref{def:u^2} for the time discretization and the difference formula for the space that the fully discrete scheme is obtained in the following
\begin{align}\label{eq: fully BDF3 implicit scheme1}
   D_3u_h^n+(1+\Delta_h)^2u_h^n +f(u_h^n)=0\quad\text{for $\mathrm{x}_h\in\Omega_h$, $1\le n\le N,$}
\end{align}
subjected to the periodic boundary conditions and the following initial condition
\begin{align}\label{eq: BDF3 scheme initial data}
   u_h^0=\varphi_0(\mathrm{x}_h)+\frac{\tau}{2}\varphi_1(\mathrm{x}_h) \quad\text{for $\mathrm{x}_h\in\bar{\Omega}_h$},
\end{align}
where $\varphi_1:=f(\varphi_0)+(1+\Delta_h)^2\varphi_0$ for the smooth data $\varphi_0\in H^4(\Omega).$

For any grid functions $v,w\in\mathcal{V}_h,$ we define the inner product
$\myinner{v,w}:=h^2\sum\limits_{{\mathrm{x}_h}\in\Omega_h}v_h w_h,$ the associated $L^2$ norm
$\mynorm{v}_{l^2}:=\sqrt{\myinner{v,v}}$ \zhur{and the discrete} $L^q$ norm
$$\mynorm{v}_{q}:=\sqrt[q]{h^2\sum_{{\mathrm{x}_h}\in\Omega_h}\abs{v_h}^q}\quad\text{for $v\in \mathcal{V}_h$}.$$
To simplify the notation, we write $\| v\|:=\|v\|_{l^2}$ and the discrete $L^{\infty}$ norm $\mynorm{v}_\infty:=\max\limits_{{\mathrm{x}_h}\in\Omega_h}|v_h|$ \zhur{in the following analysis}. The discrete Green's formula with periodic boundary conditions yields
$\myinner{-\Delta_h v,w}=\myinner{\nabla_hv,\nabla_hw}$ and $\myinner{\Delta_h^2v,w}=\myinner{\Delta_hv,\Delta_hw}.$

\section{Solvability and the energy dissipation law}
\setcounter{equation}{0}
In this section, we demonstrate the unique solvability of the BDF3 implicit scheme \eqref{eq: fully BDF3 implicit scheme1} based on the Brouwer fixed-point theorem, and then, using the positive definiteness of the convolution coefficients, the energy stability is proved in detail, which further deduces the boundedness of the numerical solution in $L^\infty$ norm.

\subsection{Unique solvability}

For simplicity, we introduce notations that $b_0^{(1)}=\frac{2}{\tau}, b_0^{(2)}=\frac{3}{2\tau}$ and $b_0^{(n)}=b_0$, for any fixed index $n$, define the map $\Pi_n:\mathcal{V}_{h}\rightarrow \mathcal{V}_{h}$ as follows
\begin{align}\label{eq: nonlinear map}
	\Pi_n(w_h)&:=b_0^{(n)}w_h-g_h^{n-1}+(1+\Delta_h)^2w_h+f(w_h),~~\mathrm{x}_h\in\bar{\Omega}_h, ~~\text{for}~~n\ge1,
\end{align}
where $g_h^{n-1}=b_0u_h^{n-1}+b_1\diff u_h^{n-1}-b_2\diff u_h^{n-2} \;\text{for}\; n\ge 3, g_h^{1}=\frac{3}{2\tau}u_h^{1}+\frac{1}{2\tau}\diff u_h^{1}$, and $g_h^0=\frac{2}{\tau}\varphi_0+\varphi_1$.

It is easily seen that the equation $\Pi_n(u^n)=0$ is equivalent to the proposed nonlinear implicit scheme \eqref{eq: fully BDF3 implicit scheme1}. Thus, the solvability of BDF3 scheme \eqref{eq: fully BDF3 implicit scheme1} can be verified via the equation $\Pi_n(w_h)=0,~n\ge 1$ in the following theorem.

\begin{theorem} Suppose the time step $\tau$ holds the condition that $\tau<\frac{3}{2(\mathrm{g}^2+\epsilon)}$, the difference scheme (\ref{eq: fully BDF3 implicit scheme1}) is uniquely solvable.
\end{theorem}
\begin{proof} {Firstly, we prove the existence of the solution.}

Suppose $u^{n-1},~u^{n-2}$ have been determined, taking the inner product of \eqref{eq: nonlinear map} with $w,$ it yields
\begin{align*}
\myinner{\Pi_n(w), w}=b_0^{(n)}\myinner{w, w}+\myinner{(1+\Delta_h)^2w, w}
+ \myinner{f(w), w}-\myinner{g^{n-1}, w},~~\text{for}~~n\ge1.
\end{align*}
For the second term, using discrete Green's formula, we obtain
\begin{align*}
\myinner{(1+\Delta_h)^2w, w}=\myinner{(1+\Delta_h)w, (1+\Delta_h)w}\ge 0.
\end{align*}
Then, it follows \zhur{that} 
\begin{align*}
\myinner{\Pi_n(w), w}\geq&\,b_0^{(n)}\|w\|^2+\|w\|_4^4-\mathrm{g}\myinner{w^2,w}-\epsilon\|w\|^2
-\|g^{n-1}\|\cdot\|w\|\\
\geq&\,\big(b_0^{(n)}-\epsilon\big)\|w\|^2+\|w\|_4^4-\big(\frac{\mathrm{g^2}}{4}\|w\|^2+\|w\|_4^4\big)-\|g^{n-1}\|\cdot\|w\|\\
=&\,\big(b_0^{(n)}-\epsilon-\frac{\mathrm{g}^2}{4}\big)\|w\|^2-\|g^{n-1}\|\cdot\|w\|.
\end{align*}  
{Evidently, it follows from the condition $\tau<\frac{3}{2(\mathrm{g}^2+\epsilon)}$ that}
$$\myinner{\Pi_n (w), w}\geq0,$$
with $\|w\|=\frac{1}{b_0^{(n)}-\epsilon-\frac{\mathrm{g}^2}{4}}\|g^{n-1}\|$. \zhur{By means of} the well-known Brouwer fixed-point theorem, there exists a $w^*$ such that
$$\Pi_n(w^*)=0,$$
which implies the BDF3 scheme (\ref{eq: fully BDF3 implicit scheme1}) is solvable.

\zhur{Then, we are prepared to prove the uniqueness of the solution.} Suppose $w_h$ and $v_h$ are the solutions of the difference scheme (\ref{eq: fully BDF3 implicit scheme1}). Denote $\zeta_h=w_h-v_h,$ \zhur{we start from} the following equation with respect to $\zeta_h,$
\begin{align}
b_0^{(n)}\zeta_h+(1+\Delta_h)^2\zeta_h+f(w_h)-f(v_h)=0.\label{eq: Error nonlinear map}
\end{align}
Taking the inner product of (\ref{eq: Error nonlinear map}) with $\zeta$, we have
\begin{align}
b_0^{(n)}\|\zeta\|^2+\|(1+\Delta_h)\zeta\|^2+\myinner{f(w)-f(v), \zeta}=0.\label{3.16}
\end{align}
\zhur{For the nonlinear term in (\ref{3.16}), the further estimation reads}
\begin{align*}
\myinner{f(w)-f(v), \zeta}&=\myinner{w^3-v^3, \zeta}-\mathrm{g}\myinner{w^2-v^2, \zeta}-\epsilon\|\zeta\|^2\\
&=\myinner{w^2+wv+v^2-\mathrm{g}w-\mathrm{g}v, \zeta^2}-\epsilon\|\zeta\|^2\\
&\ge \frac{1}{2}\myinner{(w-\mathrm{g})^2+(v-\mathrm{g})^2, \zeta^2}-\mathrm{g}^2\|\zeta\|^2-\epsilon\|\zeta\|^2.
\end{align*}
Substituting the above inequality into (\ref{3.16}), it yields
$$(b_0^{(n)}-\mathrm{g}^2-\epsilon)\|\zeta\|^2\zhur{\leq}0.$$
From $b_0^{(n)}> \mathrm{g}^2+\epsilon,$ it follows that $$\|\zeta\|=0,$$
\zhur{which implies that the scheme (\ref{eq: fully BDF3 implicit scheme1}) has a unique solution.}

\end{proof}

\subsection{Energy dissipation law}
In order to show the energy stability of the proposed scheme (\ref{eq: fully BDF3 implicit scheme1}), the following lemmas are needed.

\begin{lemma}{\rm\cite{sunBDF2}}\label{embedding-equation}
For any grid function $v\in\mathcal{V}_h,$ it has
\begin{align}
	\|v\|_\infty&\le \tilde{C}_\Omega\brab{\mynorm{v}+\mynorm{(1+\Delta_h)v}},
\end{align}
where $\tilde{C}_\Omega$ is a constant depending on the size of space domain $\Omega$ but independent of the grid size.
\end{lemma}

\begin{lemma} \label{b-scalinng}
For any real sequence $\{w_k\}_{k=3}^n$ with $n-3$ entries, it holds that
\begin{align}\label{convoluntion kernels equality1}
6\tau w_n\sum_{j=3}^n b_{n-j}w_j=
&\frac{9}{2}(w_n^2+\frac{2}{9}w_{n-1}^2)-\frac{9}{2}\left( w_{n-1}^2+\frac{2}{9}w_{n-2}^2\right)+2w_n^2+(w_n+w_{n-2})^2\\
\nonumber&+\frac{7}{2}(w_n-w_{n-1})^2, 
\qquad \qquad  for\;n\ge3.
\end{align}
\zhur{Thus, the following} discrete convolution kernels $b_{n-k}$ are positive definite,
\begin{align}\label{convoluntion kernels inequality2}
2\sum_{k=3}^n w_k \sum_{j=3}^k b_{k-j}w_j \ge 0.
\end{align}
\end{lemma}

\begin{proof}
With the definition of $b_{n-k}$ in \eqref{b_{n-k}}, we begin with a decompostion
\begin{align*}
6\tau w_n\sum_{j=3}^n b_{n-j}^{(n)}w_j
&=w_n(11w_n-7w_{n-1}+2w_{n-2}),\\
&=\frac{9}{2}(w_n^2+\frac{2}{9}w_{n-1}^2)-\frac{9}{2}\left( w_{n-1}^2+\frac{2}{9}w_{n-2}^2\right)+2w_n^2+(w_n+w_{n-2})^2\\
\nonumber&+\frac{7}{2}(w_n-w_{n-1})^2.
\end{align*}
Making use of \eqref{convoluntion kernels equality1}, we are able to prove the positive definiteness as follows
\begin{align*}
6\tau\sum_{k=3}^n w_k \sum_{j=3}^k b_{k-j}w_j
&=6\tau\left( \sum_{k=5}^n w_k \sum_{j=3}^k b_{k-j}w_j +w_4(b_1w_3+b_0w_4)+b_0w_3^2\right)\\
&\ge \frac{9}{2}(w_n^2+\frac{2}{9}w_{n-1}^2)-\frac{9}{2}(w_4^2+\frac{2}{9}w_3^2)-7w_3w_4+11w_4^2+11w_3^2\\
&=\frac{9}{2}(w_n^2+\frac{2}{9}w_{n-1}^2)-7w_3w_4+\frac{13}{2}w_4^2+10w_3^2\\
&\ge \frac{9}{2}(w_n^2+\frac{2}{9}w_{n-1}^2)-\frac{7}{2}(w_3^2+w_4^2)+\frac{13}{2}w_4^2+10w_3^2 \ge 0.
\end{align*}
\end{proof}

\begin{lemma}{\rm\cite{sunBDF2}}\label{lem:nonlinear-inequality}
For any $a,b\in \zhur{\mathbb{R}},$ the following inequality holds
$$(a^3-\mathrm{g}a^2)(a-b){ \ge} \frac{1}{4}(a^4-b^4)-\frac{\mathrm{g}}{3}(a^3-b^3)-\frac{2\mathrm{g}^2}{3}(a-b)^2.$$
\end{lemma}

\zhur{Now, we are in the position to} present the energy stability of the BDF3 formula \eqref{eq: fully BDF3 implicit scheme1}. 
Let $E[u^n]$ be the discrete version of energy functional (\ref{def:continuous energy})
$$E[u^n]:=\frac{1}{2}\mynormb{(1+\Delta_h)u^n}^2
+\frac{1}{4}\|u^n\|_4^4-\frac{\mathrm{g}}{3}\myinner{(u^n)^2, u^n}-\frac{\epsilon}{2}\|u^n\|^2.$$
Denote the modified discrete energy
\begin{align*}
\mathcal{E}[u^0]&=E[u^0],~
\mathcal{E}[u^1]=E[u^1]+ \frac{3}{4\tau}\mynormb{\diff u^1}^2,\\ 
\mathcal{E}[u^n]&=E[u^n]+ \frac{3}{4\tau}\mynormb{\diff u^n}^2+\frac{1}{6\tau}\mynormb{\diff u^{n-1}}^2,~~\text{for}~ n\ge2,
\end{align*}
\zhur{Then, the discrete energy dissipation law with respect to the above modified discrete energies is shown in the following theorem.}

\begin{theorem}\label{Energy dissipation law}
\zhur{Suppose} the time step $\tau$ satisfies
\begin{align}\label{the conditon of tau}
\tau \le \dfrac{1}{3(\frac{2\mathrm{g}^2}{3}+\frac{\epsilon }{2})},
\end{align}
the solution of the adaptive BDF3 scheme \eqref{eq: fully BDF3 implicit scheme1} satisfies
$$E[u^n]\le\mathcal{E}[u^n]\le \mathcal{E}[u^{n-1}]\le \cdots \le \mathcal{E}[u^{0}]= E[u^0],~~1\le n\le N.$$
\end{theorem}

\begin{proof} Taking the inner product of \eqref{eq: fully BDF3 implicit scheme1} with $\diff  u^n,$ it yields
\begin{align}\label{Energy-Law-Inner}
\myinner{D_3u^n, \diff  u^n}+\myinner{(1+\Delta_h)^2u^n, \diff  u^n}+\myinner{f(u^n), \diff  u^n}=0,~~n\ge3.
\end{align}
 According to Lemma \eqref{b-scalinng}, the first term on the left-hand side can be rewritten as
\begin{align*}
	6\tau\myinner{D_3u^n, \diff u^n} \ge
	\frac{9}{2}(\|\diff  u^n\|^2+\frac{2}{9}\|\diff  u^{n-1}\|^2)-\frac{9}{2}\left( \|\diff u^{n-1}\|^2+\frac{2}{9}\|\diff u^{n-2}\|^2\right)+2\|\diff u^{n}\|^2.
\end{align*}
For the second term, a direct application of the identity $2a(a-b)=a^2-b^2+(a-b)^2$ and the summation by parts gives
\zhur{\begin{align*}
\myinner{(1+\Delta_h)^2u^n,\diff  u^n} &=\myinner{(1+\Delta_h)u^n,\diff  (1+\Delta_h)u^n}\\
 &=\frac{1}{2}\braB{\|(1+\Delta_h)u^n\|^2-\|(1+\Delta_h)u^{n-1}\|^2+\|\diff((1+\Delta_h) u^n)\|^2}.
 \end{align*}}
 {For the nonlinear term in \eqref{Energy-Law-Inner}, with the help of the equality $2a(a-b)=a^2-b^2+(a-b)^2$} and Lemma \ref{lem:nonlinear-inequality}, it yields
 \begin{align*}
\myinner{f(u^n), \diff  u^n}&=\myinner{(u^n)^3-\mathrm{g}(u^n)^2, \diff u^n}-\epsilon\myinner{u^n, \diff u^n}\\
&\geq\frac{1}{4}(\|u^n\|_4^4-\|u^{n-1}\|_4^4)-\frac{\mathrm{g}}{3}\Big(\myinner{(u^n)^2, u^n}-\myinner{(u^{n-1})^2, u^{n-1}}\Big)\\
&\quad-\frac{\epsilon}{2}(\|u^n\|^2-\|u^{n-1}\|^2)-\Big(\frac{2\mathrm{g}^2}{3}+\frac{\epsilon}{2}\Big)\|\nabla_\tau u^n\|^2.
\end{align*}
As a consequence, a substitution of the above three inqualities into (\ref{Energy-Law-Inner}), results in 
\begin{align*}
\left( \frac{1}{3\tau}-\frac{2\mathrm{g}^2}{3}-\frac{\epsilon}{2}\right) \|\diff  u^n\|^2
+{ \mathcal{E}[u^n]}-{ \mathcal{E}[u^{n-1}]}\le0 .
\end{align*}
Under the condition (\ref{the conditon of tau}), it implies that $\mathcal{E}[u^n]\le \mathcal{E}[u^{n-1}]$ with $3\le n\le N.$

For $n=2$, applying Cauchy-Schwarz inequality and the condition (\ref{the conditon of tau}), we obtain
\begin{align*}
	\myinner{D_3u^2,\diff u^2}&=\frac{3}{2\tau}\|\diff u^2\|^2-\frac{1}{2\tau}\myinner{\diff u^1,\diff u^2}\\
	&\ge \frac{3}{2\tau}\|\diff u^2\|^2-\frac{1}{2\tau}\left( \frac{7}{6}\|\diff u^1\|^2+\frac{3}{14}\|\diff u^2\|^2\right)\\
	&=\frac{3}{4\tau}\left( \|\diff u^2\|^2+\frac{2}{9}\|\diff u^1\|^2\right)-\frac{3}{4\tau}\|\diff u^1\|^2+\frac{9}{14\tau}\|\diff u^2\|^2  
\end{align*}
Obviousy, it yields that $\mathcal{E}[u^2]\le \mathcal{E}[u^{1}]$.
For $n=1$, using the condition (\ref{the conditon of tau}), it has
\begin{align*}
	\myinner{D_3u^1,\diff u^1}=\frac{3}{4\tau}\|\diff u^1\|^2+\frac{5}{4\tau}\|\diff u^1\|^2=\frac{2}{\tau}\|\diff u^1\|^2,
\end{align*}
which can easily deduce $\mathcal{E}[u^1]\le \mathcal{E}[u^{0}]$.
It completes the proof.
\end{proof}

It follows from Theorem \ref{Energy dissipation law} that the solution of BDF3 scheme \eqref{eq: fully BDF3 implicit scheme1}
is bounded. \zhur{We further state the result} in the following lemma.

\begin{lemma}\label{lem:Bound-Solution}
Assume the time step $\tau$ satifies the condition \eqref{the conditon of tau}. The solution of BDF3 scheme \eqref{eq: fully BDF3 implicit scheme1} is stable in the $L^{\infty}$ norm with 
$\mynormb{u^n}_{\infty}\le c_0$ for $n\ge 1$,
where $c_0$ is independent of the time-step sizes $\tau$.
\end{lemma}
\begin{proof}
 Utilizing Theorem \ref{Energy dissipation law} and Young inequality, it follows from { $E[u^n]\le E[u^0]$} that
\begin{align}
4{ E[u^0]}\geq& 2\|(1+\Delta_h)u^n\|^2+\|u^n\|_4^4-\frac{4\mathrm{g}}{3}\myinner{(u^n)^2, u^n}-2\epsilon\|u^n\|^2\nonumber\\
\ge&\, 2\|(1+\Delta_h)u^n\|^2+\|u^n\|_4^4-\frac{2}{9}\mynormb{u^n}_4^4-2\mathrm{g}^2\|u^n\|^2-2\epsilon\|u^n\|^2\nonumber\\
\ge&\, 2\|(1+\Delta_h)u^n\|^2+\frac{7}{9}\|u^n\|_4^4-2(\mathrm{g}^2+\epsilon)\|u^n\|^2,~~ n\ge1.\label{infinite inequality}
\end{align}
For any real value $a \in R$, noticing $\big[a^2-\frac{9}{7}(1+\epsilon+\mathrm{g}^2)\big]^2\ge0,$ it yields
$$\|u^n\|_4^4\ge \frac{18}{7}(1+\epsilon+\mathrm{g}^2)\|u^n\|^2-\frac{81}{49}(1+\epsilon+\mathrm{g}^2)^2|\Omega_h|.$$
Substituting the above inequality into (\ref{infinite inequality}), it follows that 
\begin{align*}
4E[u^0]&\ge 2\|(1+\Delta_h)u^n\|^2+2\|u^n\|^2-\frac{9}{7}(1+\epsilon+\mathrm{g}^2)^2|\Omega_h|\\
&\ge\big(\|(1+\Delta_h)u^n\|+\|u^n\|\big)^2-\frac{9}{7}(1+\epsilon+\mathrm{g}^2)^2|\Omega_h|.
\end{align*}
By Lemma \ref{embedding-equation}, we have
\begin{align*}
\|u^n\|_\infty&\le \tilde{C}_\Omega(\|(1+\Delta_h)u^n\|+\|u^n\|\big)\\
&\le \tilde{C}_\Omega\sqrt{4E[u^0]+\frac{9}{7}(1+\epsilon+\mathrm{g}^2)^2|\Omega_h|}\\
&:=c_0.
\end{align*}
\end{proof}

\section{$L^2$ norm error estimate}
\setcounter{equation}{0}
\zhur{In this section, we investigate the error estimates for the difference scheme\eqref{eq: fully BDF3 implicit scheme1} and prove the $L^2$ norm covergence of the scheme.}
\subsection{the DOC kernels}
To establish the $L^2$ norm estimate, we introduce the DOC kernels $\vartheta_{n-k}$ with
\begin{align}\label{def:DOC kernel}
\vartheta_{0}:=\frac{1}{b_{0}}
\quad \mathrm{and} \quad
\vartheta_{n-k}:=-\frac{1}{b_{0}}
\sum_{j=k+1}^n\vartheta_{n-j}b_{j-k},
\quad \text{for $3\le k\le n-1$},
\end{align}
where $b_0=\frac{11}{6\tau},~b_1=-\frac{7}{6\tau},~ b_2=\frac{1}{3\tau},$ and $b_j=0$, ~for $3\le j\le n-3$.\\
In general, assume the summation $\sum_{k=i}^{j}, ~i>j$ to be zero. Obviously, rewritting the \eqref{def:DOC kernel}, a discrete orthogonal identity can be obtained
\begin{align}\label{def:DOC identity}
\sum_{j=k}^{n}\vartheta_{n-j}b_{j-k}\equiv\delta_{nk},\;3\le k\le n.
\end{align}

The following lemma presents the properties of the DOC kernels $\vartheta_{n-k},$ which plays a key role in the convergence analysis of the scheme \eqref{eq: fully BDF3 implicit scheme1}.

\begin{lemma}\label{lem: DOC property}
\begin{itemize}
  \item[(I)] For any sequences $\{w_k\}_{k=3}^n, $ we have\\
\begin{align}\label{ineq:the postive defination of DOC}
	\sum_{k=3}^{n}w_k\sum_{j=3}^{k}\vartheta_{k-j}w_j \ge 0, \qquad \text{for } 3\le n.
\end{align}
  \item[(II)] The DOC kernels $\vartheta_{n-k}$ in \eqref{def:DOC kernel} have an explicit formula
  \begin{align}\label{def:the explicit expression of theta}
  	\vartheta_{n-k}=\frac{6\tau}{11}\left[ \dfrac{39+7\sqrt{39}i}{78}\left(\dfrac{7-\sqrt{39}i}{22} \right)^{n-k}+\dfrac{39-7\sqrt{39}i}{78}\left(\dfrac{7+\sqrt{39}i}{22} \right)^{n-k} \right],\quad 
  \end{align}
  and the following bounds
\begin{align}\label{ineq:the above bound of DOC}
\displaystyle |\vartheta_{n-k}| \le(\frac{2}{11})^{\frac{n-k}{2}}\tau, \qquad
\text{and} \qquad
\displaystyle \sum_{k=3}^{n}|\vartheta_{n-k}|\le \frac{22}{9}\tau, \qquad\qquad \text{for } 3\le k\le n.
\end{align}
\end{itemize}
\end{lemma}
\begin{proof}
(1) 
Based on any given real sequence $\{w_k\}_{k=3}^n$, the sequence $\{v_k\}_{k=3}^n$ can be defined using the discrete BDF3 kernels $b_{n-k}$
\begin{align*}
v_k=-\frac{1}{b_0}\sum_{j=3}^{k-1}b_{k-j}w_j+\frac{w_k}{b_0},\qquad\text{for}\quad k\ge 3,
\end{align*}
which is equivalent to
\begin{align}\label{eq:v_k}
	w_k=\sum_{j=3}^{k}b_{k-j}v_j,\qquad\text{for}\quad k\ge 3.
\end{align}
Multiplying the DOC kernels $\vartheta_{n-k}$ on both sides of the formula\eqref{eq:v_k} and summing k form 3 to $n$, it yields
\begin{align}\label{eq:w_k}
	\sum_{k=3}^{n}\vartheta_{n-k}w_k=\sum_{k=3}^{n}\vartheta_{n-k}\sum_{j=3}^{k}b_{k-j}v_j=\sum_{j=3}^{n}v_j\sum_{k=j}^{n}\vartheta_{n-k}b_{k-j}=v_n, \qquad\text{for}\quad n\ge 3.
\end{align}
In turn, a combination of \eqref{eq:v_k} and \eqref{eq:w_k} leads to the proof of \eqref{ineq:the postive defination of DOC} 
\begin{align*}
	\sum_{k=3}^{n}w_k\sum_{j=3}^{k}\vartheta_{k-j}w_j=\sum_{k=3}^{n}v_k\sum_{j=3}^{k}b_{k-j}v_j \ge 0,\qquad\text{for}\quad n\ge 3,
\end{align*}
in which Lemma \ref{b-scalinng} is applied.\\
(2) Identity \eqref{def:DOC identity} indicates that
\begin{flalign}\label{n=k}
&\text{when}\,n=k,\qquad\quad\ \ \vartheta_{0}b_0=1, &\text{for}\quad k\ge 3,\qquad \\
\label{n=k+1}&\text{when}\,n=k+1,\qquad\vartheta_{1}b_0+\vartheta_{0}b_1=0, &\text{for}\quad k\ge 3,\qquad \\
\label{n=k+2}&\text{when}\,n\ge k+2,\qquad\vartheta_{n-k}b_0+\vartheta_{n-k-1}b_1+\vartheta_{n-k-2}b_2=0, &\text{for}\quad k\ge 3.\qquad
\end{flalign}
Substituting BDF3 kernels $b_{n-k}$ into \eqref{n=k}-\eqref{n=k+2}, the recursive formula  becomes available
\begin{align}\label{eq:the recursive formula of theta}
11\vartheta_{n-k}-7\vartheta_{n-k-1}+2\vartheta_{n-k-2}=0, \quad \text{for}\; 3\le k\le n-2,
\end{align}
where $\vartheta_{0}=\frac{6\tau}{11}$ and $\vartheta_{1}=\frac{42\tau}{121}$.\\
The explicit formula of $\{\vartheta_{n-k}\}_{k=3}^n$ is obtained by the characteristic equation method 
\begin{align}\label{characteristic equation}
&\vartheta_{n-k}=\frac{\vartheta_1-p\vartheta_0}{q-p}q^{n-k}+\frac{\vartheta_1-q\vartheta_0}{p-q}p^{n-k}, &\text{for}\quad 3\le k\le n, 
\end{align}
where $p$ and $q$ are two distinct roots of the characteristic equation $11\lambda^2
-7\lambda+2=0$. Furthermore, taking absolute values for $\vartheta_{n-k}$, it leads to
\begin{align*}
|\vartheta_{n-k}| &=\left|\frac{6\tau}{11}\left[ \dfrac{39+7\sqrt{39}i}{78}\left(\dfrac{7-\sqrt{39}i}{22} \right)^{n-k}+\dfrac{39-7\sqrt{39}i}{78}\left(\dfrac{7+\sqrt{39}i}{22} \right)^{n-k} \right]\right|\\
&\le\frac{12\tau}{11} \sqrt{(\frac{39}{78})^2+(\frac{7\sqrt{39}}{78})^2}\left(\sqrt{(\frac{7}{22})^2+(\frac{\sqrt{39}}{22})^2}\right)^{n-k} \\
&=\tau \sqrt{\frac{96}{143}}\left( \frac{2}{11}\right)^{\frac{n-k}{2}} \\
&\le \left( \frac{2}{11}\right)^{\frac{n-k}{2}}\tau, \qquad\qquad\qquad\text{for}\quad 3\le k \le n.
\end{align*}
Summing $k$ from 3 to $n$ separately for the both sides of the above inequality, it yields
\begin{align*}
\sum_{k=3}^{n}|\vartheta_{n-k}|&\le \tau\sum_{k=3}^{n}\left( \frac{2}{11}\right)^{\frac{n-k}{2}} \le \tau\sum_{k=3}^{\infty}\left( \frac{2}{11}\right)^{\frac{n-k}{2}}\le\frac{\tau}{1-\sqrt{\frac{2}{11}}} \le \frac{22}{9}\tau.
\end{align*}
It completes the proof.
\end{proof} 

\subsection{Convergence analysis}

\zhur{For the regularity of the exact solution $U^n,$ suppose there exists a constant $c_1>0$ such that $\{\|U^n\|_\infty, \|\partial_t U^n\|_\infty\}\le c_1,$ where $c_1$ is a constant
independent of the spatial step $h$ and the time step $\tau.$}
\begin{lemma}{\rm\cite{Liao_MA2020}}\label{lem: Gronwall inequality}
Let $\lambda\geq0,$ the sequences $\{p_k\}_{k=1}^N$ and $\{V_k\}_{k=1}^N$ be nonnegative. If the sequence satisfies
$$V_n\le \lambda\sum_{j=1}^{n-1}\tau_jV_j+\sum_{j=1}^{n}p_j,~~for~ 1\le n\le N,$$
then it holds that
$$V_n\le \exp(\lambda t_{n-1})\sum_{j=1}^{n}p_j, ~~for~ 1\le n\le N.$$
\end{lemma}

Now we are ready to show the convergence analysis of the BDF3 difference scheme\eqref{eq: fully BDF3 implicit scheme1}.
Denote the error $ e_h^n=U_h^n-u_h^n,~{ \mathrm{x}}_h\in\bar{\Omega}_h,~ 0\le n\le N.$
The error equations is shown as follows
 \begin{align}\label{error scheme}
 &{ D}_3e_h^n+(1+\Delta_h)^2e_h^n+\big(f(U_h^n)-f(u_h^n)\big)=\xi_h^n+\eta_h^n,~~{ \mathrm{x}}_h\in\bar{\Omega}_h,~1\le n\le N, 
 \end{align}
with $e_h^0=0, ~{ \mathrm{x}}_h\in \bar{\Omega}_h,$ where $\xi_h^n$ and $\eta_h^n$ denote the local consistency error in time and space, respectively.
 \begin{theorem}\label{th3.2}
Suppose the problem (\ref{def:the SH equation}) has a unique smooth solution and $u_h^n\in\mathcal{V}_h$ is the solution of the difference scheme  (\ref{eq: fully BDF3 implicit scheme1}). If the time step $\tau$ satisifies the condition that
\begin{align}\label{the condition of tau}
	\tau\le \frac{9}{88\rho},
\end{align}
then, the numerical scheme is convergent in $L^2$ norm,
$$ \|e^k\|\le {C}\bra{\tau^3+h^2},$$
where $\rho =c_1^2+c_1c_0+c_0^2+\mathrm{g}(c_0+c_1)+\epsilon,$ C is a positive constant independent of the time steps $\tau$ and the space length $h$.
\end{theorem}

\begin{proof} 
For $n=1$, taking the discrete inner product with \eqref{error scheme} by $e^1$ leads to 
\begin{align}\label{e1}
	\myinnerb{D_3e^1,e^1}+\myinnerb{(1+\Delta_h)^2e^1,e^1}+\myinnerb{(f(U^1)-f(u^1),e^1)}=\myinnerb{\xi^1+\eta^1,e^1},
\end{align}
where $\xi_h^1=O(\tau^2)$, $\eta_h^1=O(h^2).$ 

For the nonlinear term in (\ref{e1}), denote $$f_h^1=(U_h^1)^2+U_h^1u_h^1+(u_h^1)^2-\mathrm{g}(U_h^1+u_h^1)-\epsilon,$$ then it yields that $f(U_h^1)-f(u_h^1)=f_h^1e_h^1.$ It follows from Lemma \ref{lem:Bound-Solution} and  $\|U^n\|_\infty \le c_1$ that $$\|{ f^1}\|_\infty\le \rho ,$$ where $\rho =c_1^2+c_1c_0+c_0^2+\mathrm{g}(c_0+c_1)+\epsilon.$	

Noticing  $\myinnerb{(1+\Delta_h)^2e^1,e^1}=\|(1+\Delta)e^1\|^2 \ge 0$ with $e^0=0$, and using Cauchy-Schwarz inequality, \eqref{e1} can be rewritten as
\begin{align*}
	\frac{2}{\tau}\|e^1\|^2-\rho \|e^1\|^2 \le (\|\xi^1\|+\|\eta^1\|)\|e^1\|,
\end{align*}
that is 
\begin{align*}
	(2-\rho\tau)\|e^1\|\le \tau(\|\xi^1\|+\|\eta^1\|).
\end{align*}
In view of the condition \eqref{the condition of tau}, it yields
\begin{align}\label{e11}
	\|e^1\|\le C_1(\tau^3+h^2),
\end{align}
where $C_1$ is a positive constant independent of $\tau$ and $h$.

While for $n=2$, taking the inner product of \eqref{error scheme} with $e^2$, we obtain
\begin{align}\label{e2}
	\myinnerb{D_3e^2,e^2}+\myinnerb{(1+\Delta)^2e^2,e^2}+\myinnerb{(f(U^1)-f(u^2),e^2)}=\myinnerb{\xi^2+\eta^2,e^2},
\end{align}
where $\xi_h^2=O(\tau^2)$, $\eta_h^2=O(h^2).$ 
An application of Cauchy-Schwarz inequality for the first term on the left hand side directly leads to the following inquality
\begin{align*}
	\myinnerb{D_3e^2,e^2}=\frac{3}{2\tau}\myinnerb{e^2,e^2}-\frac{2}{\tau}\myinnerb{e^1,e^2} \ge \frac{3}{2\tau}\|e^2\|^2-\frac{2}{\tau}\|e^1\|\|e^2\|.
\end{align*}
Similar to the case of $n = 1$, equation \eqref{e2} can be rewritten as
\begin{align*}
	(\frac{3}{2}-\rho\tau)\|e^2\|\le 2\|e^1\|+\tau(\|\xi^1\|+\|\eta^1\|).
\end{align*}
According to the condition \eqref{the condition of tau}, the following convergence estimate is obtained
\begin{align}\label{e22}
	\|e^2\|\le C_2(\tau^3+h^2),
\end{align}
where $C_2$ is a positive constant independent of $\tau$ and $h$.

For $n \ge 3$, replacing $n$ by $l$ in (\ref{error scheme}) and multiplying both sides of (\ref{error scheme}) by the DOC kernels $\vartheta_{k-l},$ and summing $l$ from 3 to $k,$ for ${ \mathrm{x}}_h\in\bar{\Omega}_h,$ we have the following identity for $ k \ge 3$
\begin{align}
\sum_{l=3}^k\vartheta_{k-l}D_3e_h^l+\sum_{l=3}^k\vartheta_{k-l}(1+\Delta_h)^2 e_h^l
-\sum_{l=3}^k\vartheta_{k-l}\brab{f(U_h^l)-f(u_h^l)}=P_h^k+\zhur{Q_h^k}, \label{Error-Equation-DOC}
\end{align}
where $P_h^k=\sum\limits_{l=3}^k\vartheta_{k-l}\xi_h^l$ and $\zhur{Q_h^k}=\sum\limits_{l=3}^k\vartheta_{k-l}\eta_h^l.$

It follows from Lemma \ref{lem: DOC property} (II) and $\xi_h^l=O(\tau^3), \eta_h^l=O(h^2)$ that
\begin{align}
\sum_{k=3}^n\|P^k\| &\le \frac{22}{9}t_n\tau^3 \le  C_3\tau^3,\;\;{\rm for}\; 3\le n\le N,\label{Pn}\\
\sum_{k=3}^n\|\zhur{Q^k}\| &\le \frac{22}{9}t_n\tau^3 \le C_4h^2,\;\;{\rm for}\; 3\le n\le N,\label{Qn}
\end{align}
where $C_3, C_4$ are positive constants independent of $\tau$ and $h$.

\zhur{Recalling} the orthogonal identity \eqref{def:DOC identity}, it yields
\begin{align*}
\sum_{l=3}^k\vartheta_{k-l}D_3e_h^l&=\sum_{l=3}^k\vartheta_{k-l}\sum_{j=1}^lb_{l-j}\diff e_h^j
=\sum_{l=3}^k\vartheta_{k-l}\left( \sum_{j=3}^lb_{l-j}\diff e_h^j+b_{l-1}\diff e_h^1+b_{l-2}\diff e_h^2\right) \\
&=\sum_{j=3}^{k}\diff e_h^j\sum_{l=j}^{k}\vartheta _{k-l}b_{l-j}+\sum_{l=3}^k\vartheta_{k-l}b_{l-1}\diff e_h^1+\sum_{l=3}^k\vartheta_{k-l}b_{l-2}\diff e_h^2\\
&=\diff e_h^k+\sum_{l=3}^k\vartheta_{k-l}b_{l-1}\diff e_h^1+\sum_{l=3}^k\vartheta_{k-l}b_{l-2}\diff e_h^2, \qquad\text{for} ~3\le k\le n,
\end{align*}
where the orthogonal identy \eqref{def:DOC identity} is applied.\\
Taking the inner product of (\ref{Error-Equation-DOC}) with $e^k$ and summing up $k$ from $3$ to $n$, it yields
\begin{align}
 &\sum_{k=3}^n\myinnerb{\nabla_\tau e^k, e^k}+\sum_{k=3}^n\sum_{l=3}^k\vartheta_{k-l}^{(k)}\myinnerb{(1+\Delta_h)^2e^l, e^k}
+\sum_{k=3}^n\sum_{l=3}^k\vartheta_{k-l}^{(k)}\myinnerb{f(U^l)-f(u^l), e^k}\nonumber\\
&=-\sum_{k=3}^{n}\sum_{l=3}^k\vartheta_{k-l}b_{l-1}\myinnerb{\diff e^1,e^k}-\sum_{k=3}^{n}\sum_{l=3}^k\vartheta_{k-l}b_{l-2}\myinnerb{\diff e^2,e^k}+\sum_{k=3}^n\myinnerb{P^k+\zhur{Q^k}, e^k}.\label{inner eq}
\end{align}
\zhur{In the following, each term in \eqref{inner eq} will be estimated rigorously.
For the first term on the left hand side, the following identity holds}
\begin{align}
\sum_{k=3}^n\myinnerb{\nabla_\tau e^k, e^k}=\frac{1}{2}(\|e^n\|^2-\|e^2\|^2+\sum_{k=3}^n\|\diff e^k\|^2).\label{error-first}
\end{align}
For the second \zhur{one} on the left hand side, with the application of the positive definiteness of the discrete kernels $\vartheta_{n-k}$ in \eqref{ineq:the postive defination of DOC},  we have
\begin{align}\label{error-second}
\sum_{k=3}^n\sum_{l=3}^k\vartheta_{k-l}\myinnerb{(1+\Delta_h)^2e^l, e^k}
=\sum_{k=3}^n\sum_{l=3}^k\vartheta_{k-l}\myinnerb{(1+\Delta_h)e^l, (1+\Delta_h)e^k}\ge0.
\end{align}
For the third nonlinear term on the left-hand side, we begin with a estimate
$$f(U_h^l)-f(u_h^l)=f_h^le_h^l,$$
where we denote $$f_h^l=(U_h^l)^2+U_h^lu_h^l+(u_h^l)^2-\mathrm{g}(U_h^l+u_h^l)-\epsilon.$$ 
It follows from Lemma \ref{lem:Bound-Solution} that $$\|{ f^l}\|_\infty\le \rho ,$$ where $\rho=c_1^2+c_1c_0+c_0^2+\mathrm{g}(c_0+c_1)+\epsilon.$
Consequently, it yields
\begin{align}\label{error-third}
\sum_{k=3}^n\sum_{l=3}^k\vartheta_{k-l}\myinnerb{f(U^l)-f(u^l), e^k}
=\sum_{k=3}^n\sum_{l=3}^k\vartheta_{k-l}\myinnerb{f^le^l, e^k}\le \rho\sum_{k=3}^n \|e^k\|\sum_{l=3}^k|\vartheta_{k-l}|\|e^l\|.
\end{align}

 Noticing $b_j=0$ when $j\ge 3,$ the first term on the right hand side of (\ref{inner eq}) gives
\begin{align}\label{error-forth}
-\sum_{k=3}^{n}\sum_{l=3}^k\vartheta_{k-l}b_{l-1}\myinnerb{\diff e^1,e^k}&=-\sum_{k=3}^{n}\vartheta_{k-3}b_{2}\myinnerb{\diff e^1,e^k}\le \frac{1}{3\tau}\sum_{k=3}^{n}|\vartheta_{k-3}|\|\diff e^1\|\|e^k\|,
\end{align}
where Cauchy-Schwarz inequality is used. Similarly, the second term on the right hand side of (\ref{inner eq}) can be estimated as
\begin{align}\label{error-fifth}
	-\sum_{k=3}^{n}\sum_{l=3}^k\vartheta_{k-l}b_{l-2}\myinnerb{\diff e^2,e^k}&=-\sum_{k=3}^{n}\vartheta_{k-3}b_{1}\myinnerb{\diff e^2,e^k}-\sum_{k=4}^{n}\vartheta_{k-4}b_{2}\myinnerb{\diff e^2,e^k}\nonumber\\
	&\le \frac{7}{6\tau}\sum_{k=3}^{n}|\vartheta_{k-3}|\|\diff e^2\|\|e^k\|+\frac{1}{3\tau}\sum_{k=4}^{n}|\vartheta_{k-4}|\|\diff e^2\|\|e^k\|.
\end{align}
Substituting (\ref{error-first})-(\ref{error-fifth}) into (\ref{inner eq}), one arrives at
\begin{align}
&\|e^n\|^2-\|e^2\|^2 \le 2\rho\sum_{k=3}^n \|e^k\|\sum_{l=3}^k|\vartheta_{k-l}|\|e^l\|+\frac{2}{3\tau}\sum_{k=3}^{n}|\vartheta_{k-3}|\|\diff e^1\|\|e^k\|\nonumber\\&\;\;
+\frac{7}{3\tau}\sum_{k=3}^{n}|\vartheta_{k-3}|\|\diff e^2\|\|e^k\|+\frac{2}{3\tau}\sum_{k=4}^{n}|\vartheta_{k-4}|\|\diff e^2\|\|e^k\|
+2\sum_{k=3}^n \|e^k\|(\|P^k\|+\|\zhur{Q^k}\|),\label{error-inequality}
\end{align}
where the last item is obtained by Cauchy-Schwarz inequality.

\zhur{There exists an} integer $n_0$ such that $\|e^{n_0}\|=\max\limits_{0\le k\le n}\|e^k\|.$ Taking $n=n_0$ in (\ref{error-inequality}), we have
\begin{align*}
\|e^{n_0}\|&\le \|e^2\|+2\rho\sum_{k=3}^{n_0} \|e^k\|\sum_{l=3}^k|\vartheta_{k-l}|+\frac{2}{3\tau}\sum_{k=3}^{n_0}|\vartheta_{k-3}|\|\diff e^1\|\nonumber\\
&+\frac{7}{3\tau}\sum_{k=3}^{n_0}|\vartheta_{k-3}|\|\diff e^2\|+\frac{2}{3\tau}\sum_{k=4}^{n_0}|\vartheta_{k-4}|\|\diff e^2\|
+2\sum_{k=3}^{n_0}(\|P^k\|+\|\zhur{Q^k}\|), 
\end{align*}
\zhur{When} $\tau\le \frac{9}{88\rho}$, it follows from Lemma \ref{lem: DOC property} { (II)} that
\begin{align}
\|e^n\|\le\|e^{n_0}\|&\le 2\|e^2\|+\frac{88}{9}\rho\tau\sum_{k=3}^{n-1}\|e^k\|+\frac{88}{27}\|\diff e^1\|\nonumber\\ &\quad+\frac{44}{3}\|\diff e^2\|
+4\sum_{k=3}^{n}(\|P^k\|+\|\zhur{Q^k}\|), ~~1\le n\le N.
\end{align}
 Applying {Lemma \ref{lem: Gronwall inequality}, and combining the estimates \eqref{e11} and \eqref{e22} leads to} the following overall estimation
\begin{align*}
\|e^n\|&\le 2\exp(\frac{88}{9}\rho t_{n-1})\braB{\|e^2\|+\frac{44}{27}\|\diff e^1\|+\frac{22}{3}\|\diff e^2\|+2\sum_{k=3}^{n}(\|P^k\|+\|\zhur{Q^k}\|)}\\
&\le 2\exp(\frac{88}{9}\rho t_{n-1})\braB{\|e^2\|+\frac{44}{27}\|e^1\|+\frac{22}{3}\sqrt{2}(\|e^2\|+\|e^1\|)+2C_3\tau^3+2C_4h^2}\\
&\le  C(\tau^3+h^2),\qquad\qquad\qquad\qquad\qquad 3\le n\le N.
\end{align*}
where C is a positive constant independent of the time step $\tau$ and the space length $h$.
\end{proof}

\section{Numerical results}
In this section, numerical examples in 2D and 3D are used to verify the convergence order in time direction, global energy stability and evolution properties of the numerical solution. The $L^2$ norm error between the exact solution $U^N_h$ and the numerical solution $u^N_h$ is denoted by $e(N):=\|U^N-u^N\|,$ and the order of convergence in temporal direction is estimated by {$\mathrm{order=log}(e(N)/e(2N)).$}
 
\begin{example}{(Temporal Accuracy Test)}
	Consider the Swift-Hohenberg equation $u_t+(1+\Delta)^2u+f(u)=g(x,y,t),\;(x,y)\in\Omega,\;0< t\le T,$ in the two-dimensional domain $\Omega=(0,2\pi)^2$ with $T=10$, subjected to the periodic boundary conditions and the initial data
	$$ u(x,y,0)=\sin(2x)\sin(2y),~~(x,y) \in \Omega,$$
	such that it has an exact solution 
	$$u(x,y,t)=\cos(t)\sin(2x)\sin(2y).$$
\end{example}

This example is also used in \cite{sunBDF2} for the numerical test.  The spatial grid points are fixed to be M=10000, and the temporal steps are varied by N=10,20,30,40. Model parameters $\mathrm{g}$ and $\epsilon$ are specified in different cases that (1)$\epsilon=0.25,~\mathrm{g}=1,$ (2)$\epsilon=0.15, ~\mathrm{g}=1$ and (3)$\epsilon=0.25,~\mathrm{g}=0$ to test the convergence of  the method under different conditions. The experimental results are displayed in Table \ref{table:order}, which shows the order of convergence in time direction matches well with the theoretical analysis in Section 4. 

\begin{table}[htb!]
	\begin{center}
	\caption{$L^2$ errors and convergence orders of BDF3 scheme \eqref{eq: fully BDF3 implicit scheme1} with different $\epsilon$ and $\mathrm{g}$. } 
	\vspace*{0.3pt}
	\def\temptablewidth{1\textwidth}
	{\rule{\temptablewidth}{1pt}}
	\begin{tabular*}{\temptablewidth}{@{\extracolsep{\fill}}lllllll}
		\hline
		\multicolumn{1}{c}{}  & \multicolumn{2}{l}{$\epsilon$=0.25, $g=$1} & \multicolumn{2}{l}{$\epsilon$=0.15, $\mathrm{g}$=1} & \multicolumn{2}{l}{$\epsilon$=0.25, $\mathrm{g}$=0} \\ \cline{2-3}\cline{4-5}\cline{6-7} 
		\multicolumn{1}{c}{N} & e(N)              & Order         & e(N)           & Order           & e(N)          & Order          \\ \hline
		10                    & 3.167e-01          & *             & 3.154e-01      & *               &  3.127e-01           &  *	              \\
		20                    & 5.699e-02          & 2.47          & 5.647e-02      & 2.48            &  5.315e-02             & 2.56              \\
		40                    & 7.321e-03          & 2.96          & 7.246e-03      & 2.96            & 6.625e-03             & 3.00              \\
		80                    & 9.311e-04          & 2.98          & 9.309e-04      & 2.96            & 8.443e-04             & 2.97	              \\ \hline
	\end{tabular*}
	{\rule{\temptablewidth}{1pt}}
		\label{table:order}
	\end{center}
\end{table}

\begin{example}{(Evolution of energy and solution in 2D)}
We consider the Swift-Hohenberg equation \eqref{def:the SH equation} on $\Omega=(0,100)^2$ with different parameters $\epsilon$ and $\mathrm{g}$, subjected to the periodic boundary conditions and an initial data\cite{Su_CAM2019}
$$u(x,y,0)=0.1+0.02\cos\left( \frac{\pi x}{100}\right)\sin\left( \frac{\pi y}{100}\right)+0.05\sin\left( \frac{\pi x}{20}\right)\cos\left( \frac{\pi y}{20}\right), ~~~~~(x,y)\in \Omega.$$
\end{example}

\begin{figure}[htb!]
	\centering
	\subfigure[$\mathrm{g}$=0.5]{
		\includegraphics[width=3.0in]{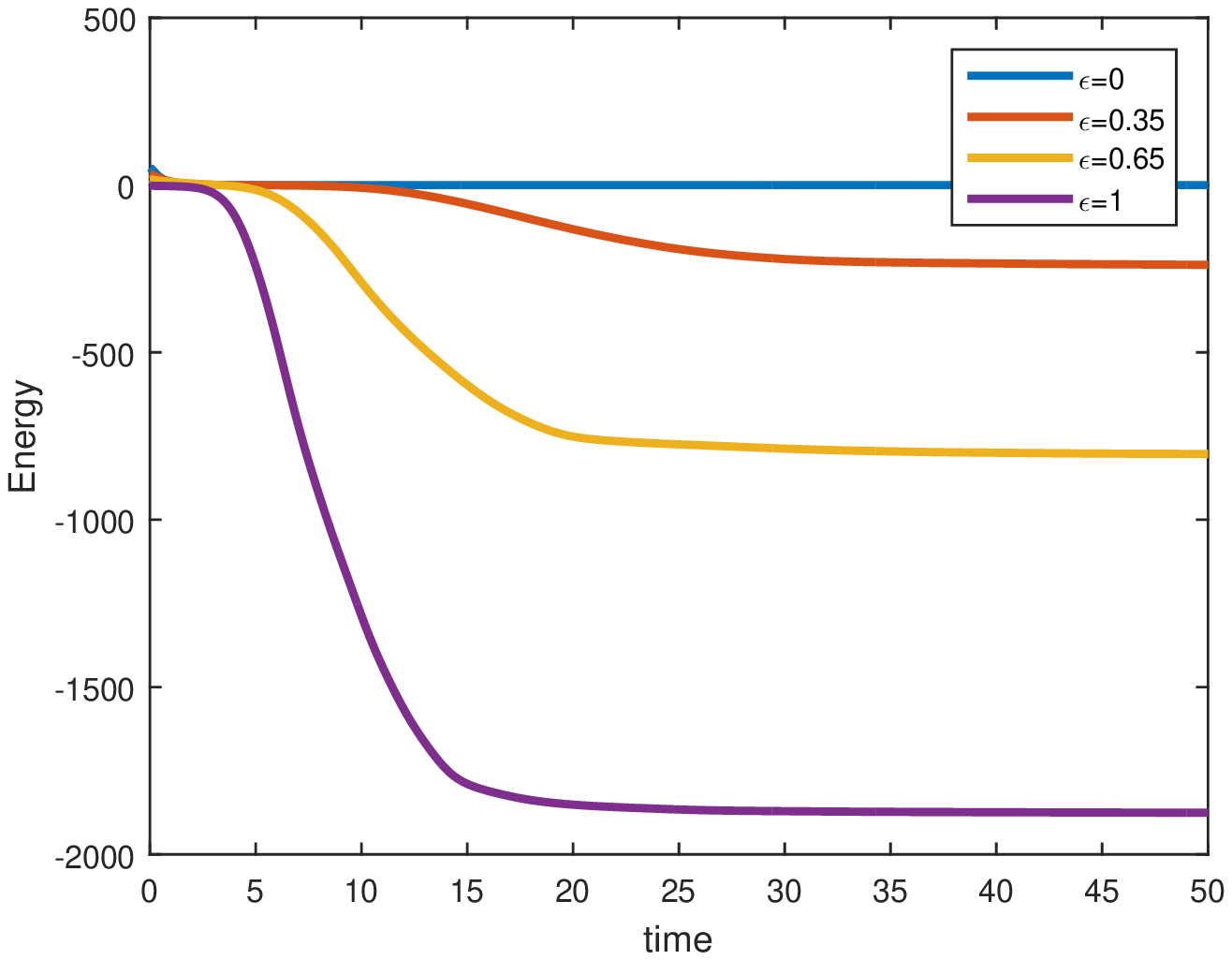}}
	\subfigure[$\epsilon$=0.5]{
		\includegraphics[width=3.0in]{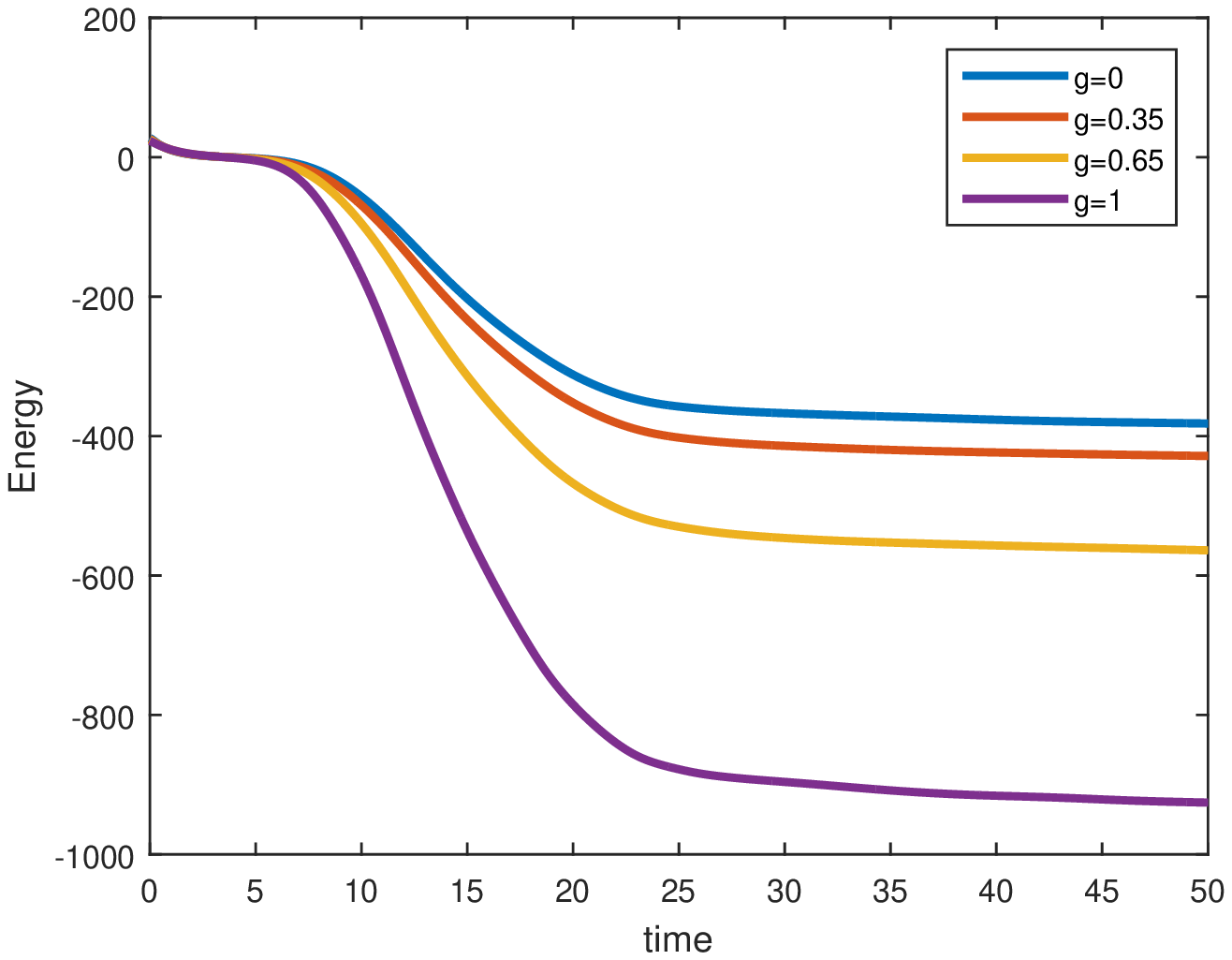}}\\
	\caption{ Evolution of energy with different parameters $\epsilon$ and $\mathrm{g}$.}
	\label{figure-energy_1}
\end{figure}
In this example, we will show the evolution of both the energy and the numerical solution based on spatial meshes $128\times128$ with $\tau=0.1$. The energy evolution in time with $T=50$ for different parameters $\epsilon$ and $\mathrm{g}$ are shown in Figure \ref{figure-energy_1}, from which we see that scheme \eqref{eq: fully BDF3 implicit scheme1} is always energy dissipating for any parameters as tested. What's more, the values of $\epsilon$ and $\mathrm{g}$ appear to manipulate the decay rate of the energy. 

The parameters $\epsilon$ and $\mathrm{g}$ have great impact on the dissipation property according to Figure \ref{figure-energy_1}. For fixed $\tau=0.1$ and $\epsilon=0.25,$ Figure \ref{snapshot-figureg0} and Figure \ref{snapshot-figureg1} show the numerical solutions before stabilization at different time $t=0.1,16,32,64$ with parameters $\mathrm{g}=0$ and $\mathrm{g}=1$ respectively. We find that the evolved patterns of the numerical solution are effectd by the parameter $\mathrm{g}$, which reveals cylindrical patterns when $\mathrm{g}=1$ and curvilinear patterns when $\mathrm{g}=0$ at the saturation time.

\begin{figure}[htb!]
	\centering
	\subfigure[t=0.1]{
		\includegraphics[width=1.2in]{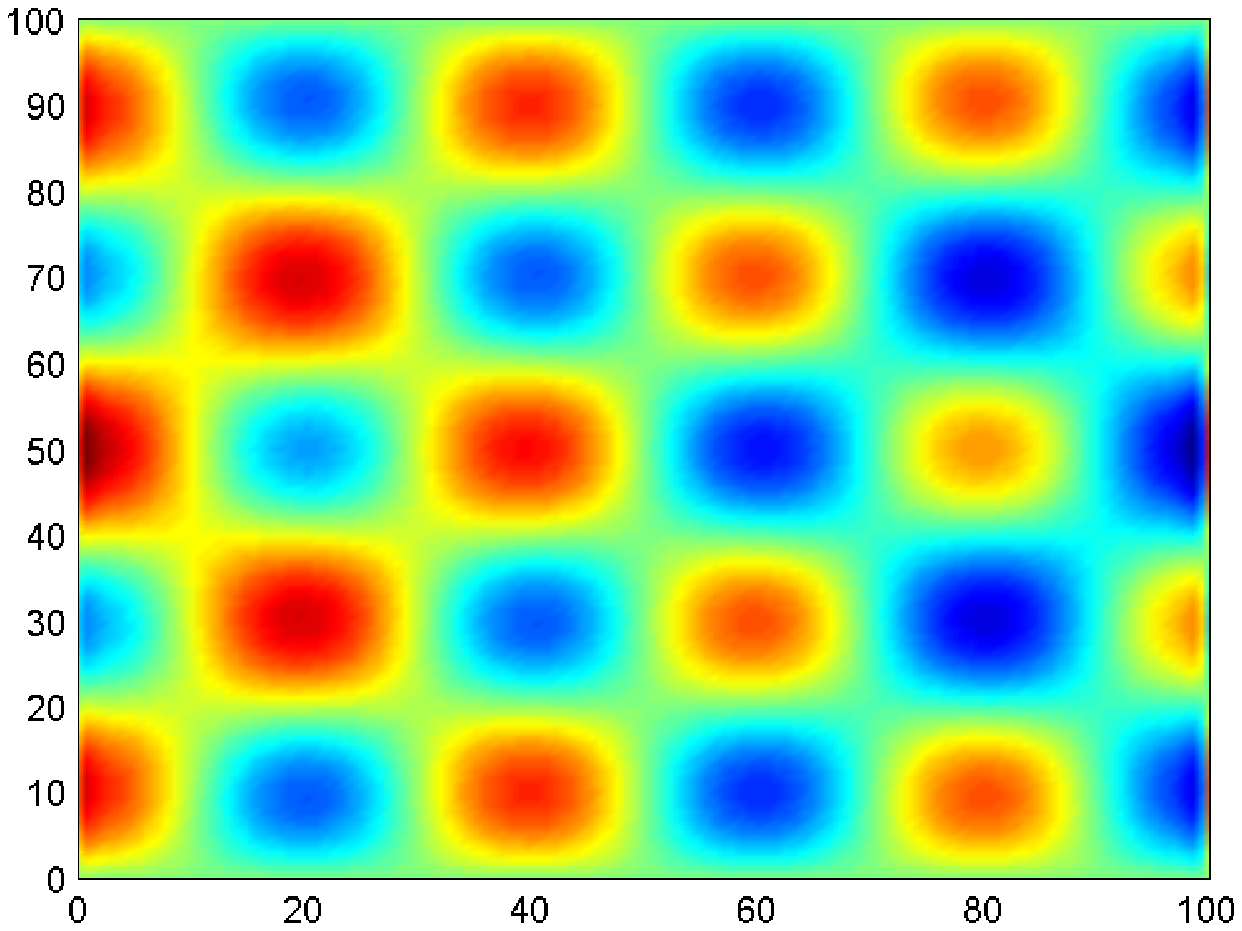}}
	\subfigure[t=16]{
		\includegraphics[width=1.2in]{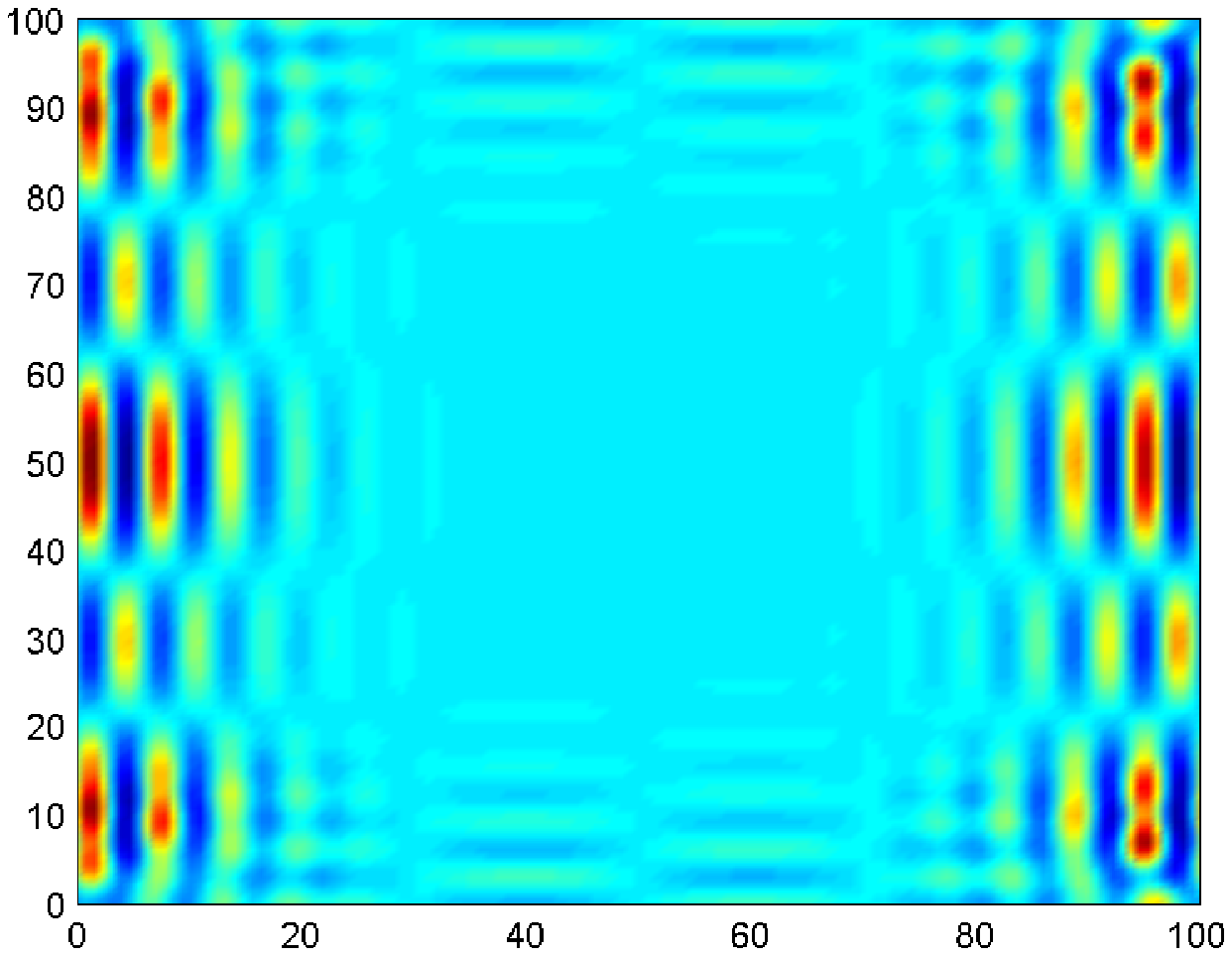}}
	\subfigure[t=32]{
		\includegraphics[width=1.2in]{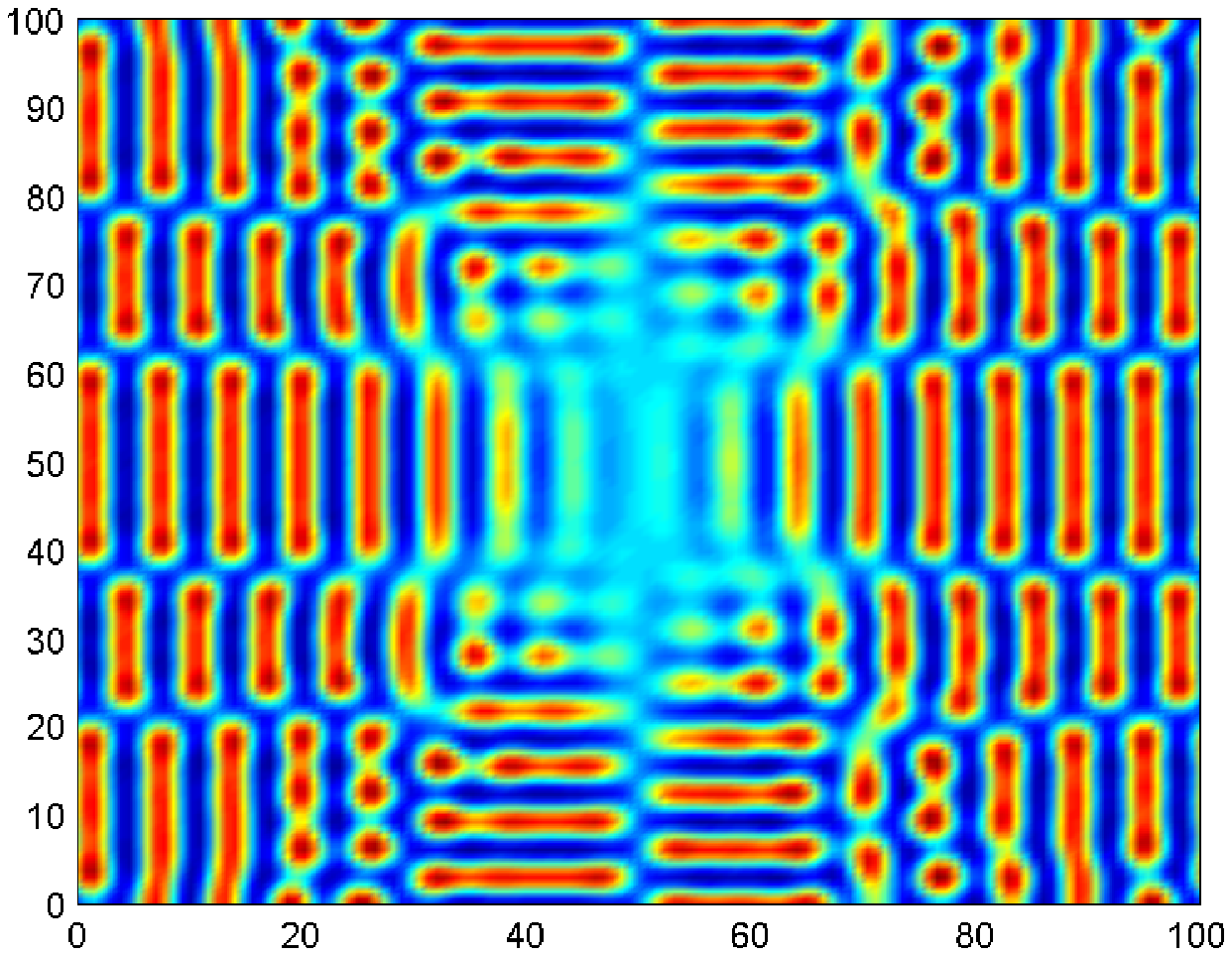}}
	\subfigure[t=64]{
		\includegraphics[width=1.2in]{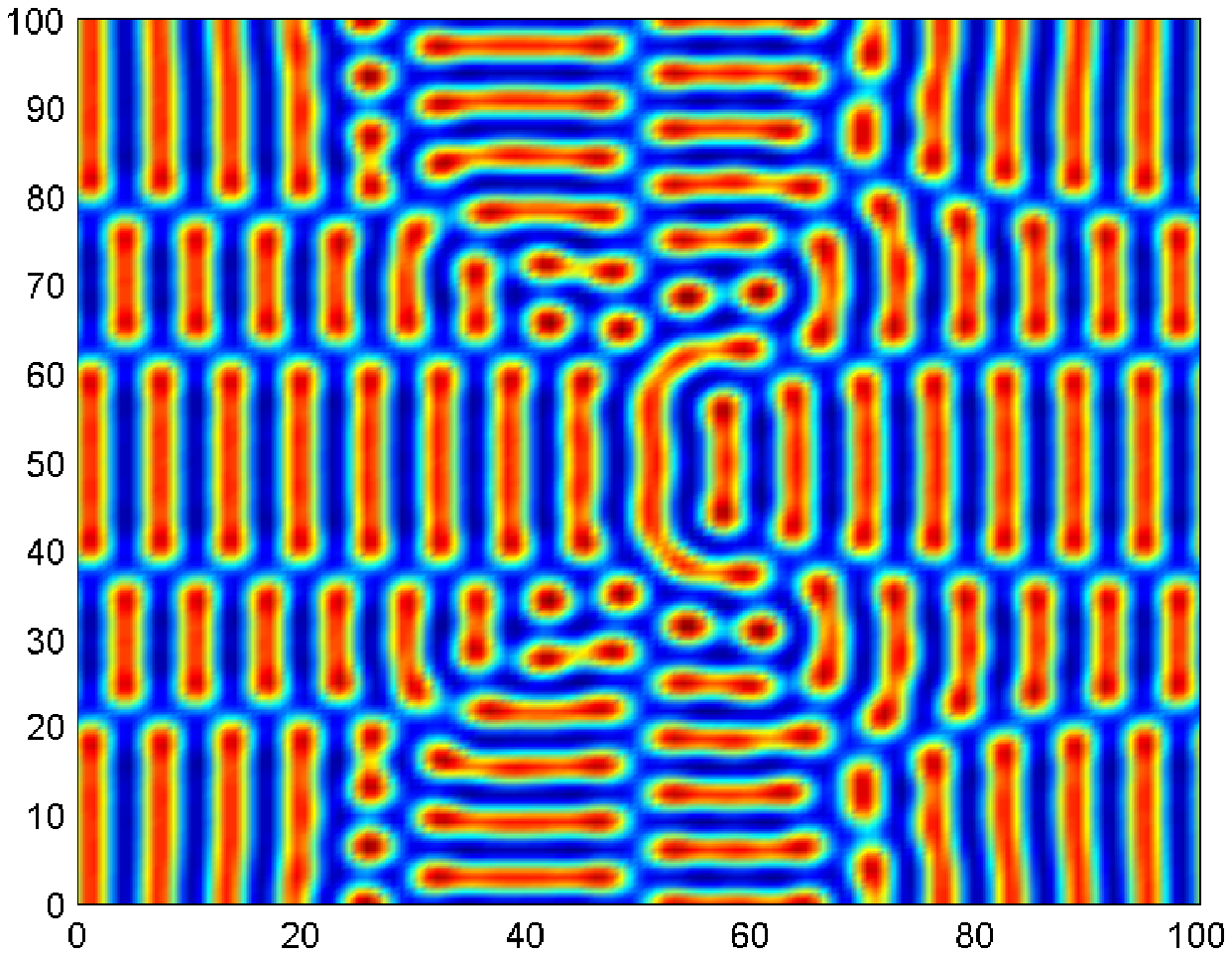}}
	\caption{Evolution of $u(x,y,t)$ at $t=0.1, 16, 32, 64$ with $\mathrm{g}=1$}\label{snapshot-figureg0}
\end{figure}

\begin{figure}[htb!]
	\centering
	\subfigure[t=0.1]{
		\includegraphics[width=1.2in]{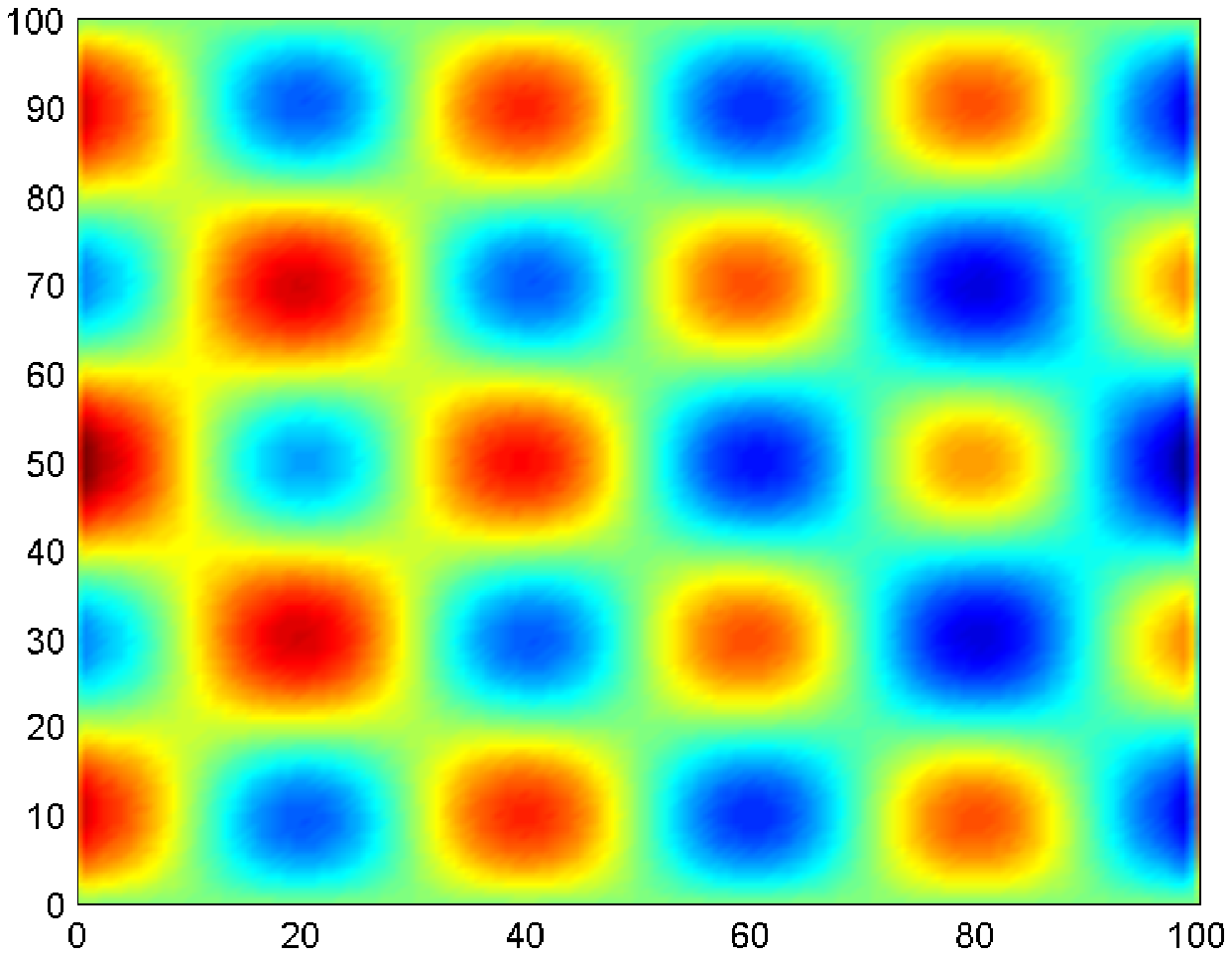}}
	\subfigure[t=16]{
		\includegraphics[width=1.2in]{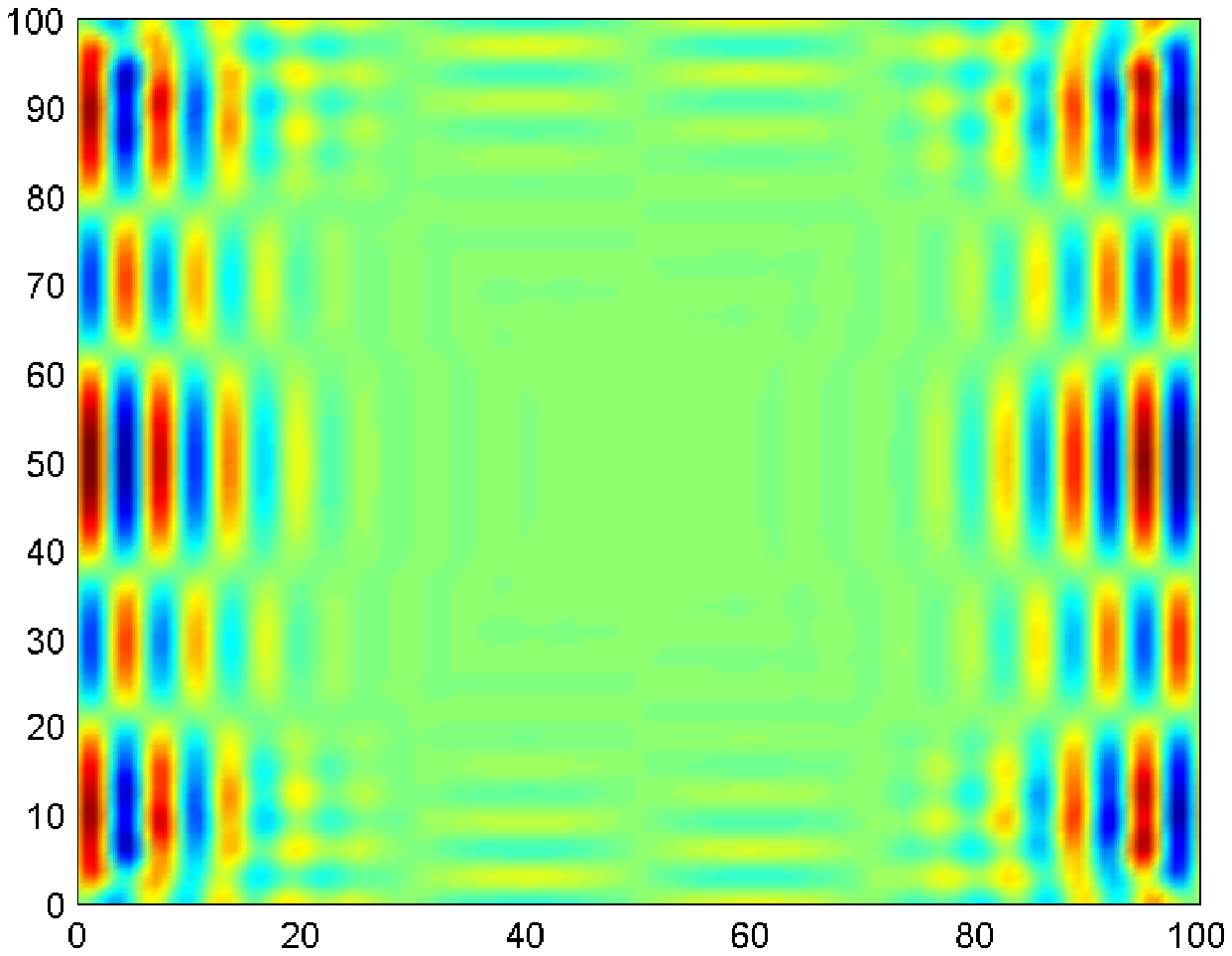}}
	\subfigure[t=32]{
		\includegraphics[width=1.2in]{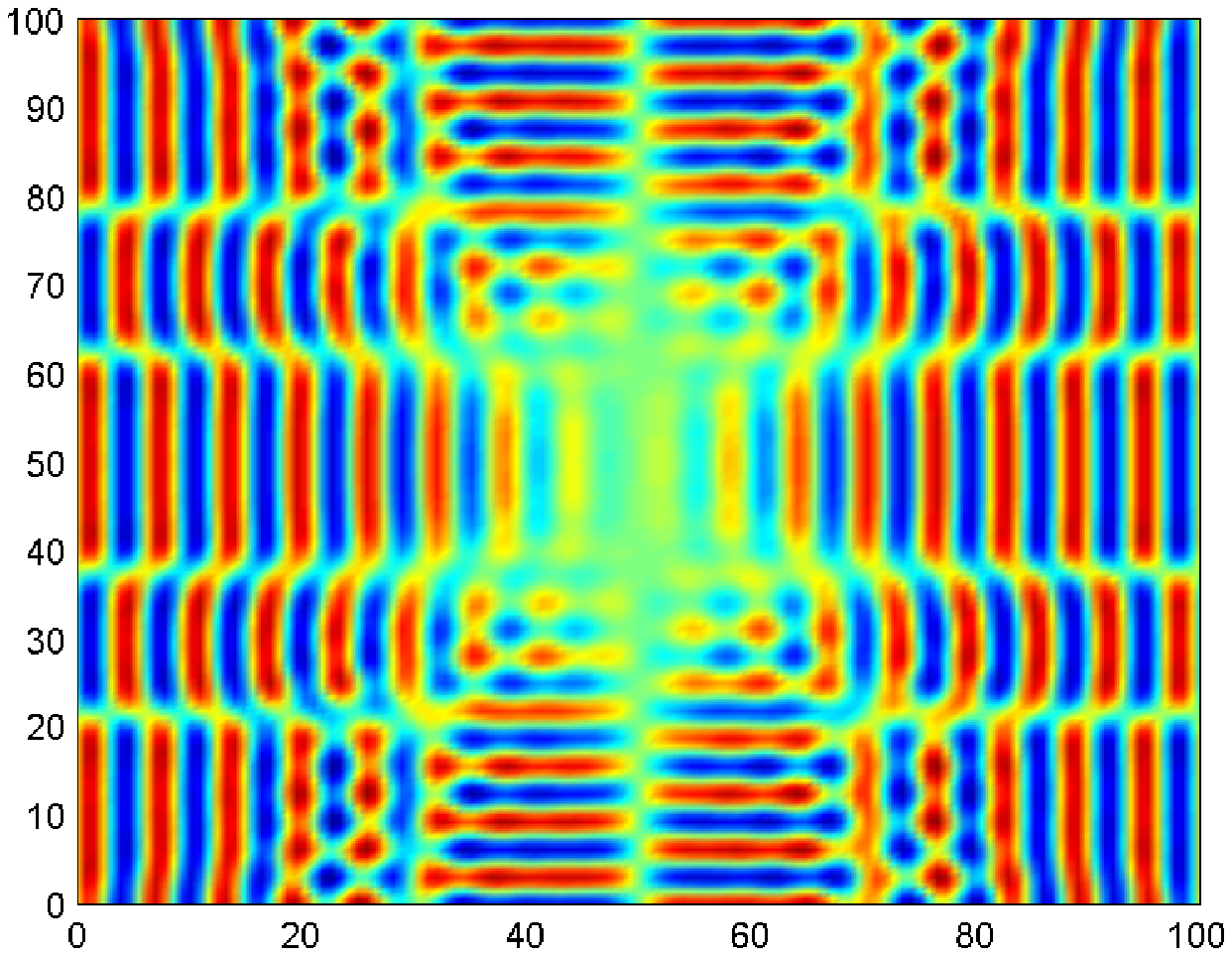}}
	\subfigure[t=64]{
		\includegraphics[width=1.2in]{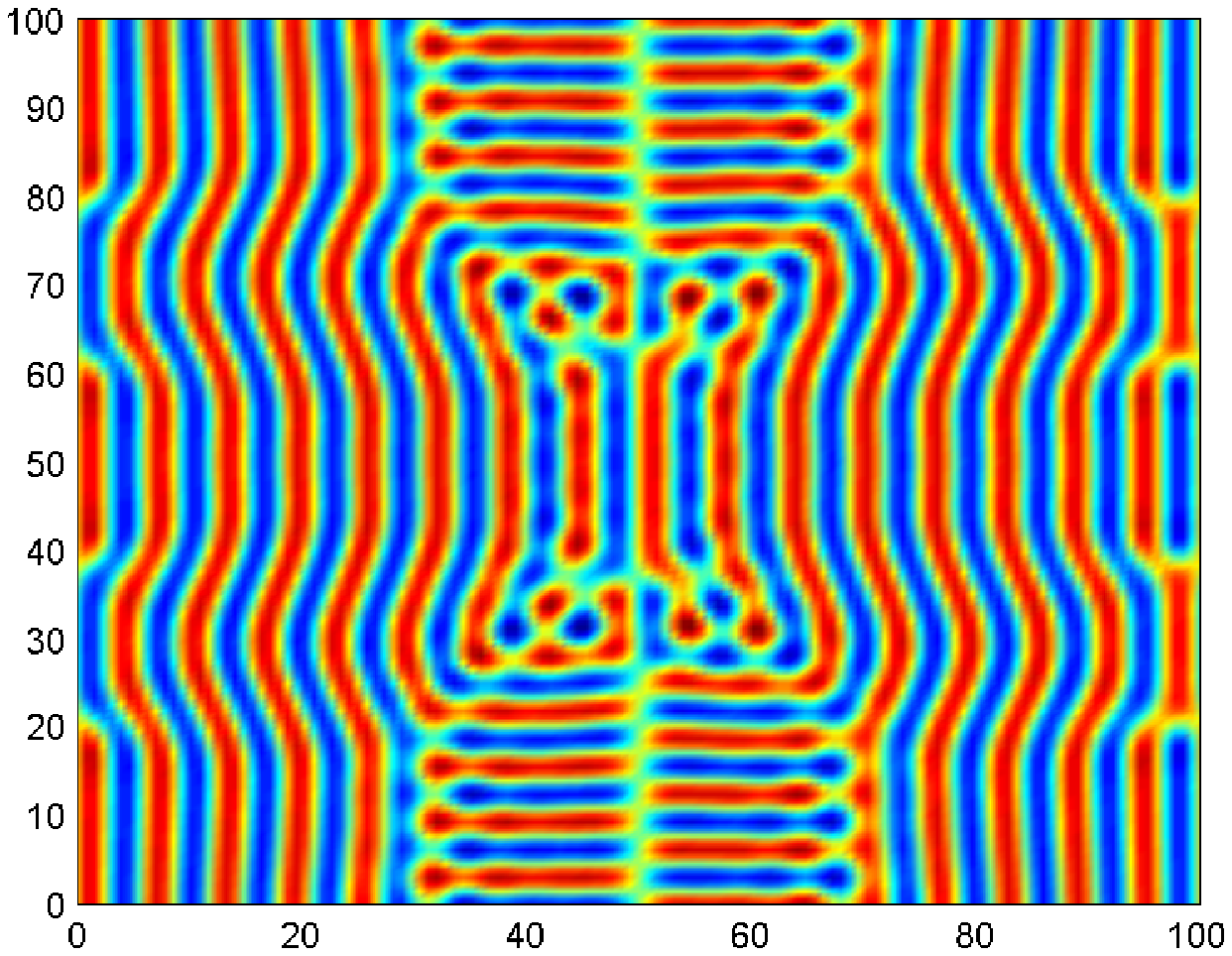}}
	\caption{Evolution of $u(x,y,t)$ at $t=0.1, 16, 32, 64$ with $\mathrm{g}=0$}\label{snapshot-figureg1}
\end{figure}

\begin{example}{(Evolution of energy and solution in 3D)}
	In this example, in order to further test the evolutionary properties of the BDF3 scheme \eqref{eq: fully BDF3 implicit scheme1},	we consider the Swift-Hohenberg equation \eqref{def:the SH equation} on three-dimensional space $\Omega=(0,48)^3$ with  $\epsilon=0.1,~\mathrm{g}=0$, subjected to the periodic boundary conditions and random initial condition
	$$u(x,y,z,0)=0.01*rand(x,y,z),~~~~~(x,y,z)\in \Omega,$$
	where rand(x,y,z) is a random number between $-0.01$ and $0.01$.
\end{example}

For fixed $h=1,$ $\tau=0.01,$ $\epsilon=0.1,~\mathrm{g}=0,$ the numerical solutions are reported in Figure \ref{snapshot-figureg3}, while the snapshots from $t=0.01$ to $t=5000$ reveal vividly the formation and evolution of the curvilinear patterns. The pattern evolution looks slow in the beginning, however, we observe that at a certain point, before $t=100$ in this case, the states of circular aggregation break up giving way to the curvilinear patterns. A stable curvilinear pattern is taking its shape after $t\ge 5000$, and the steady state is approached. The energy evolution in Figure \ref{figure-energy_11} clearly confirms this.

\begin{figure}[htb!]
	\centering
	\subfigure[t=0.01]{
		\includegraphics[width=2in]{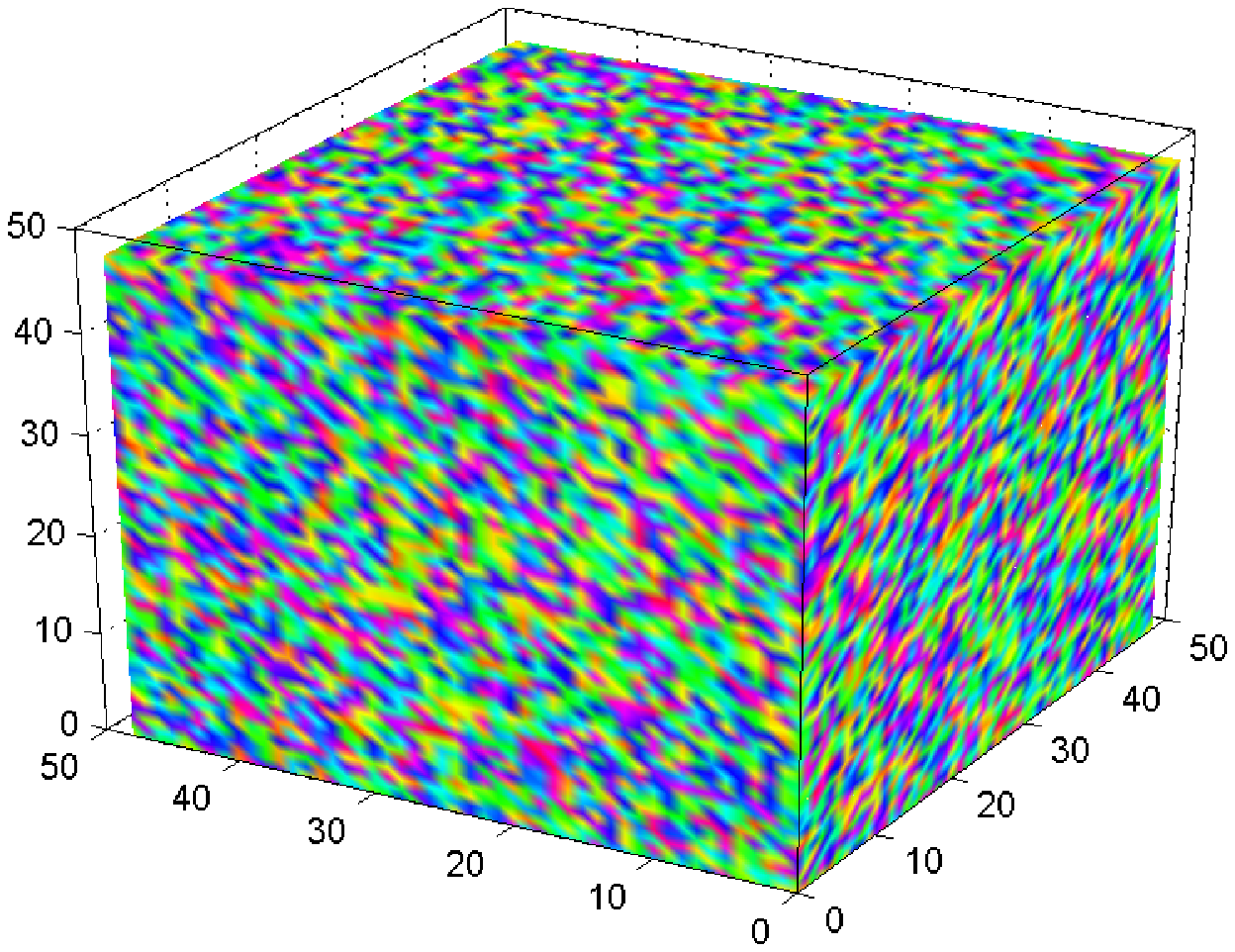}}
	\subfigure[t=10]{
		\includegraphics[width=2in]{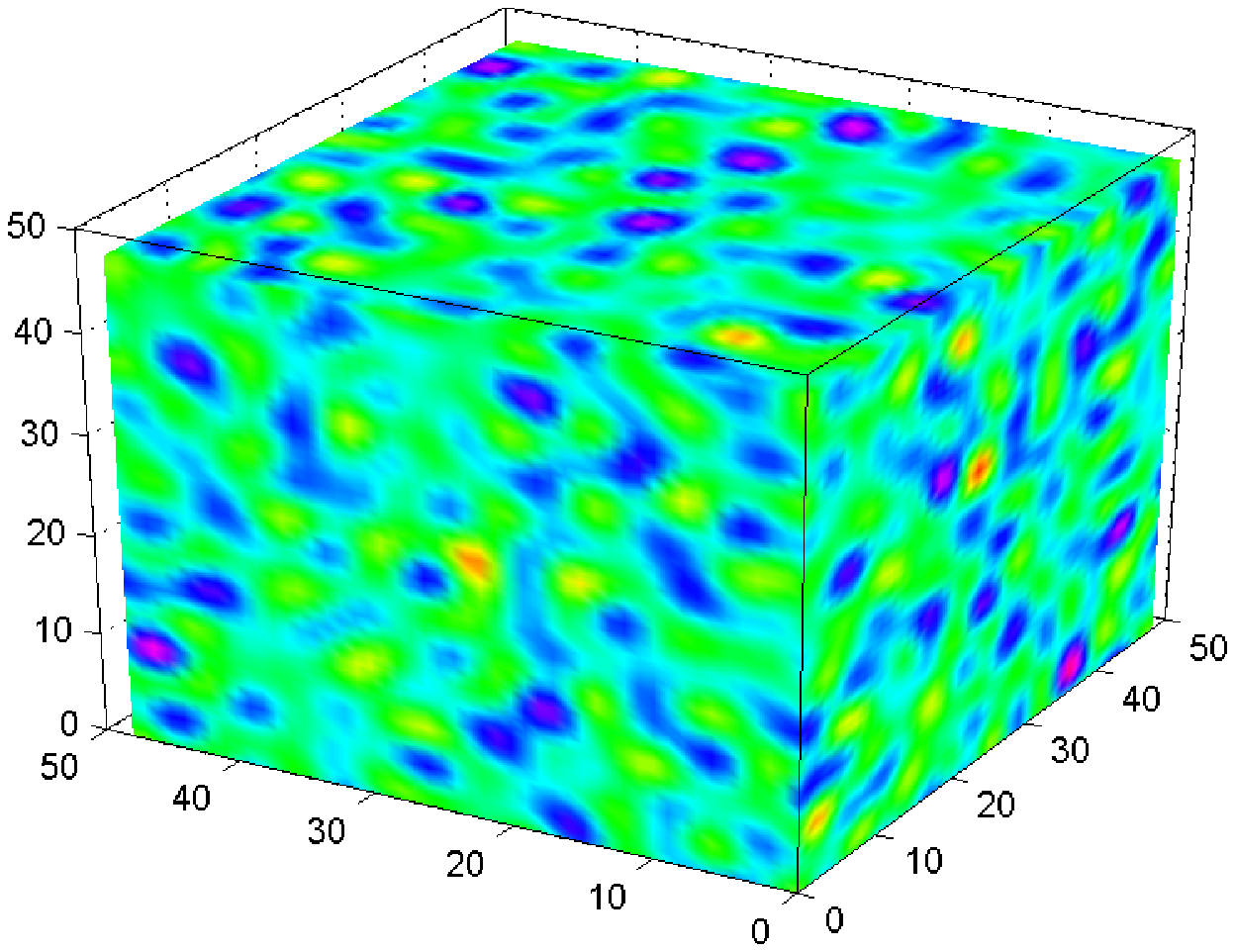}}
	\subfigure[t=100]{
		\includegraphics[width=2in]{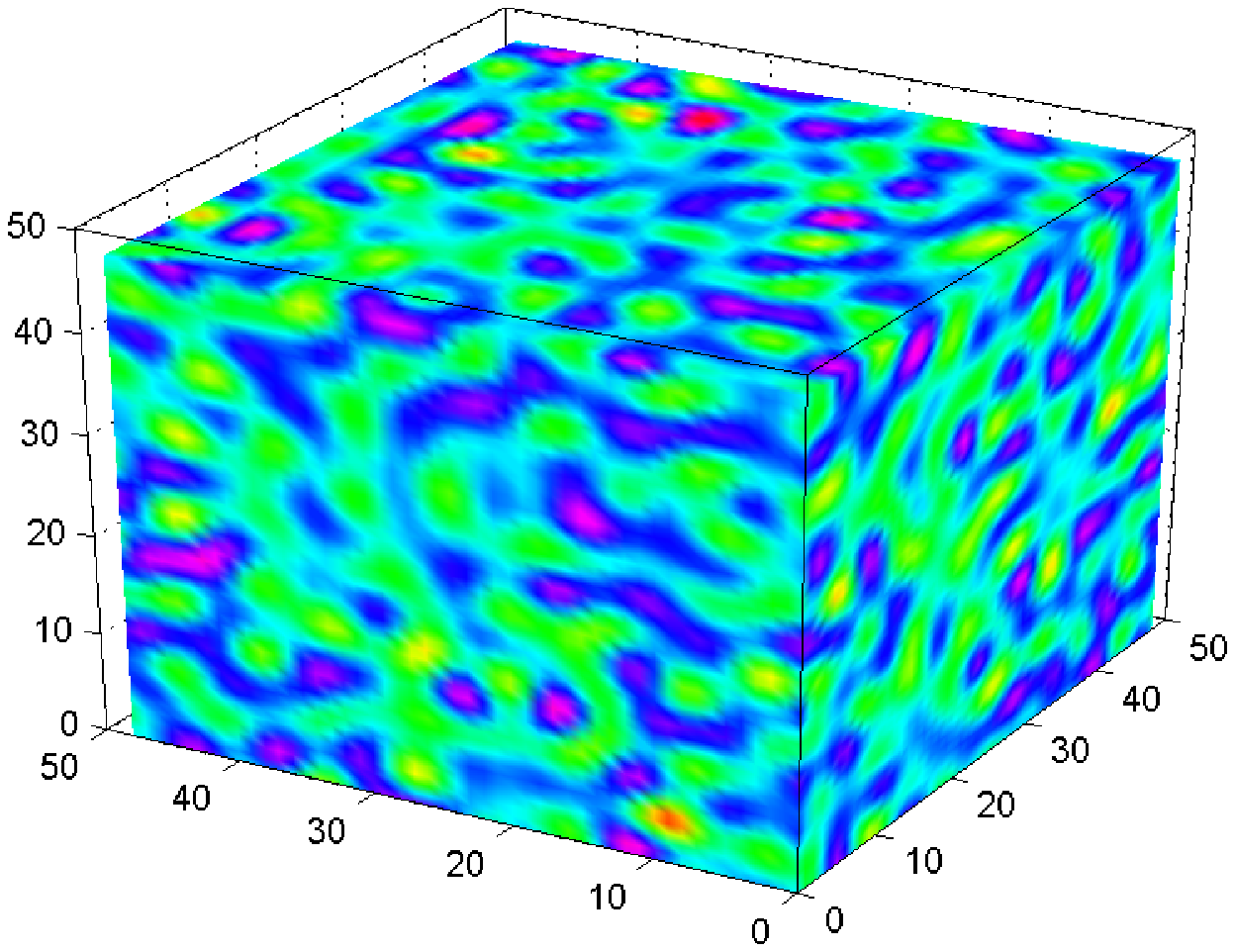}}\\
	\subfigure[t=500]{
		\includegraphics[width=2in]{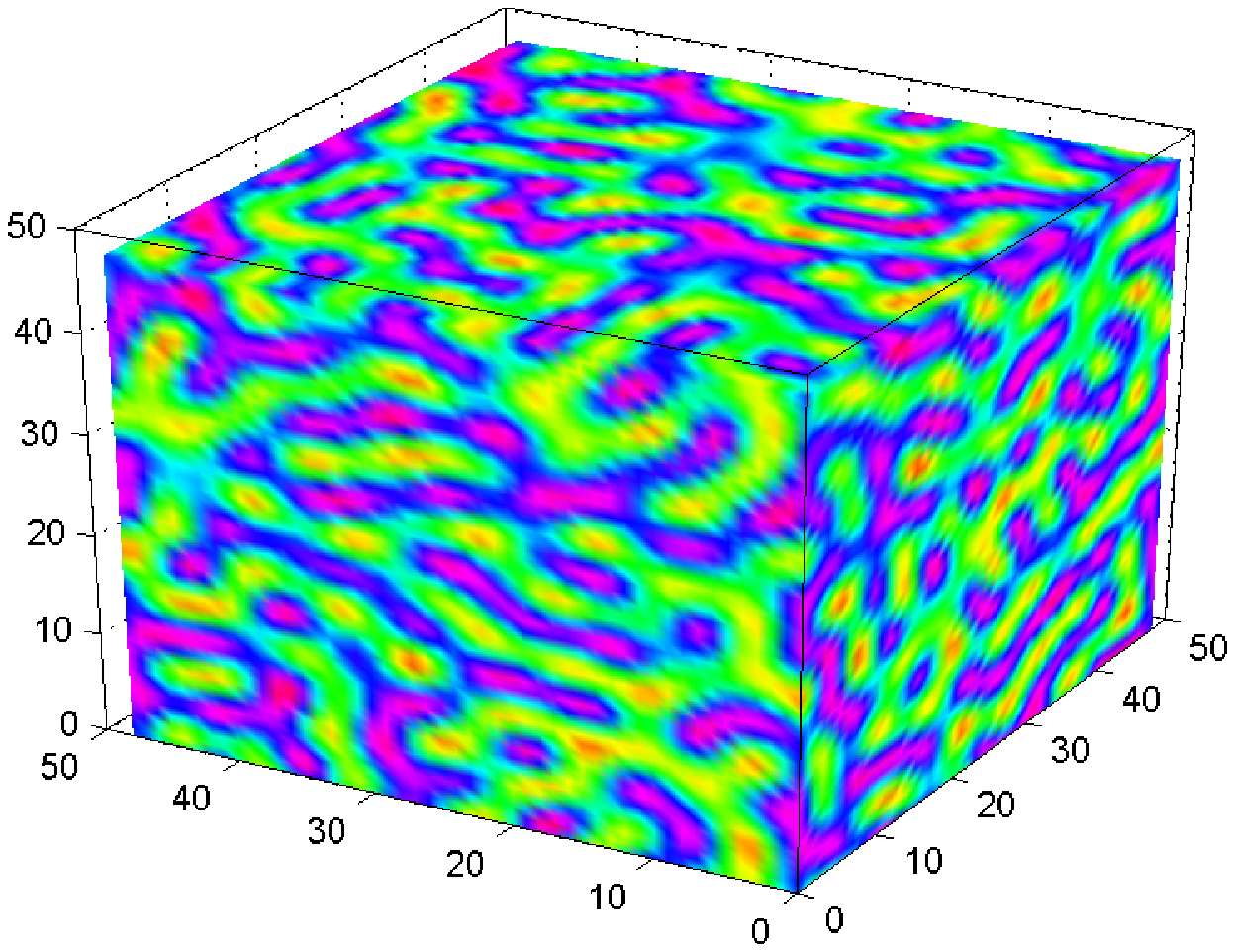}}
	\subfigure[t=2000]{
		\includegraphics[width=2in]{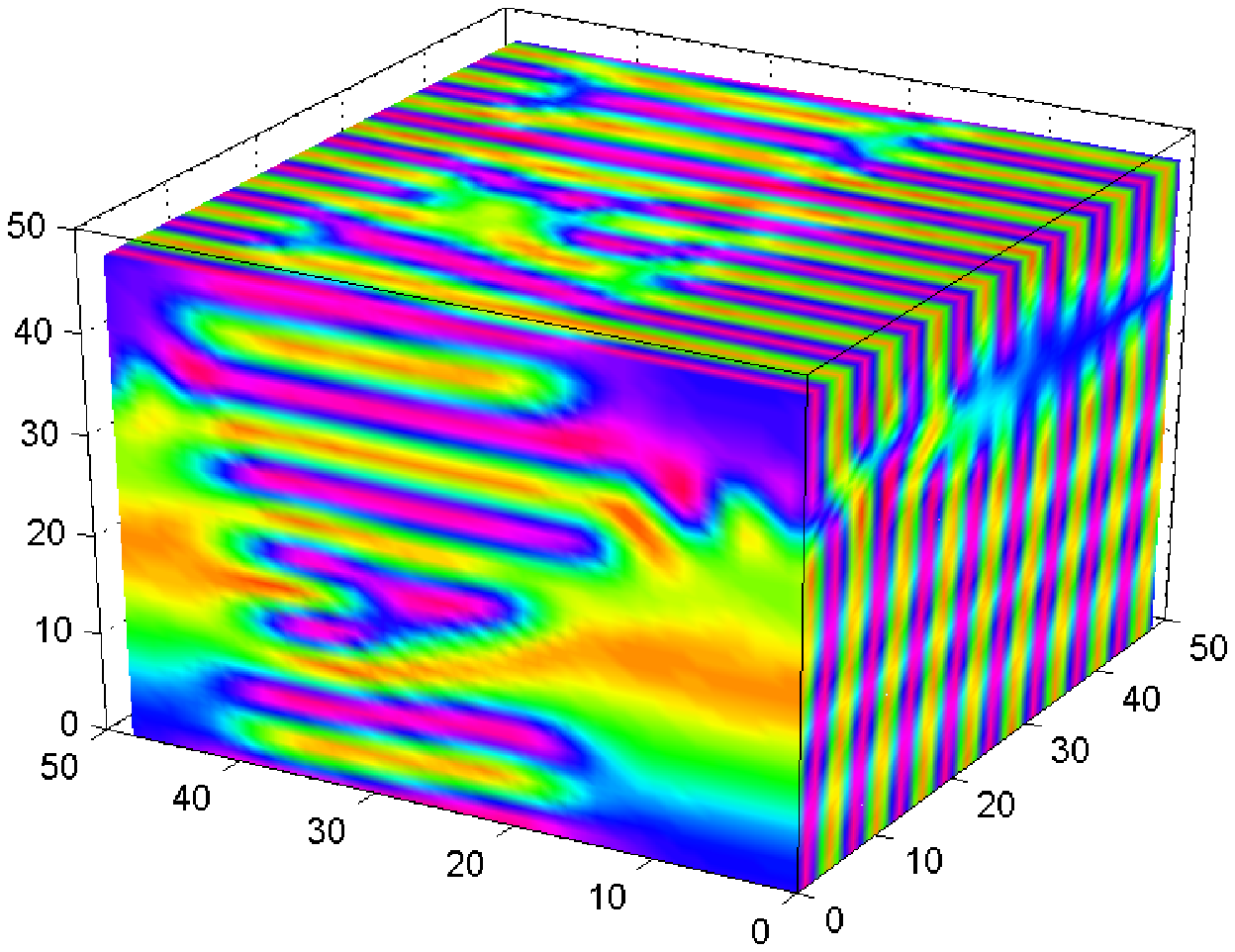}}
	\subfigure[t=5000]{
		\includegraphics[width=2in]{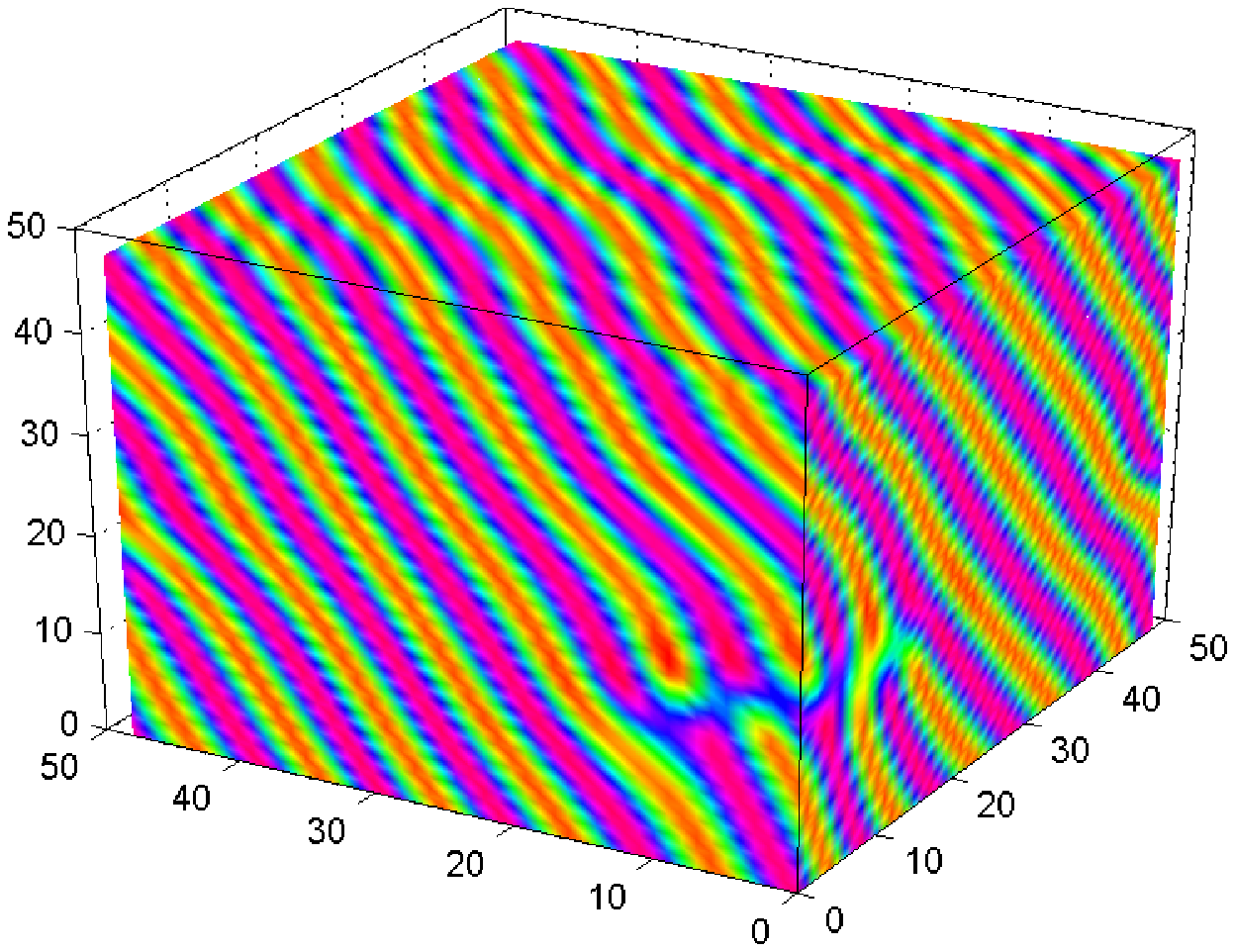}}

	\caption{Evolution of $u(x,y,t)$ at $t=0.1, 10, 100,500,2000,5000$ with $\mathrm{g}=0$}\label{snapshot-figureg3}
\end{figure}

\begin{figure} [htb!]
	\centering
	\includegraphics[width=3.0in]{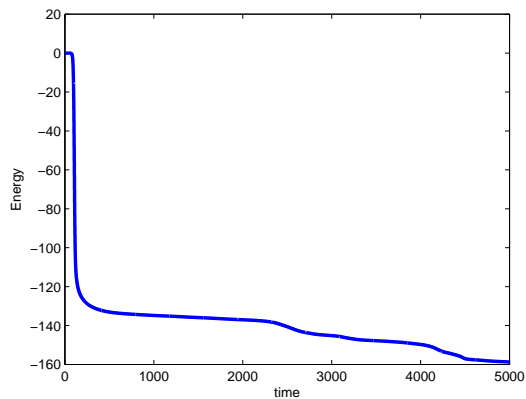}
	\caption{Evolution of energy using the BDF3 scheme (\ref{eq: fully BDF3 implicit scheme1}).}\label{figure-energy_11}
\end{figure}

\section{Concluding remarks}
In this paper, we presented and analyzed an energy stable, uniquely solvable, and convergent numerical scheme for the Swift-Hohenberg equation. 
The energy stability and unique solvability of the fully discrete scheme were derived using the Brouwer fixed-point theorem and some analytical techniques. 
The $L^2$ norm convergence was then proved utilizing the DOC kernel. Numerical results showed the convergence order and dissipative properties of the proposed BDF3 scheme in 2D and 3D simulations. The analysis framework developed in this work can be further extended to some other problems with gradient flow structure.

\section*{Acknowledgement}
\noindent We would like to acknowledge support by the National Natural Science Foundation of China (No. 11701081,11861060), the Fundamental Research Funds for the Central Universities, the Jiangsu Provincial Key Laboratory of Networked Collective Intelligence (No. BM2017002), Key Project of Natural Science Foundation of China (No. 61833005) and ZhiShan Youth Scholar Program of SEU, China Postdoctoral Science Foundation (No. 2019M651634), High-level Scientific Research foundation for the introduction of talent of Nanjing Institute of Technology (No. YKL201856).

\end{document}